\documentclass[11pt]{article}
\usepackage{hyperref}
\usepackage[active]{srcltx}
\usepackage[curve]{xypic}
\usepackage{graphicx}
\usepackage{longtable}
\usepackage{rotating}
\usepackage{multirow}
\usepackage[active]{srcltx}
\usepackage{epsfig,amsmath}
\usepackage[english]{babel}
\usepackage{amsmath,amsfonts,amssymb,amscd,color,epsfig,amsthm}
\usepackage{titletoc}
\usepackage{mathrsfs}
\usepackage{appendix}
\usepackage{enumitem}
\usepackage{tikz,tikz-cd}
\setenumerate[1]{itemsep=0pt,partopsep=0pt,parsep=\parskip,topsep=5pt}
\setitemize[1]{itemsep=0pt,partopsep=0pt,parsep=\parskip,topsep=5pt}
\setdescription{itemsep=0pt,partopsep=0pt,parsep=\parskip,topsep=5pt}
\usepackage{caption}

\newtheorem{theorem}{Theorem}[section]
\newtheorem{lemma}{Lemma}[section]
\newtheorem{proposition}{Proposition}[section]

\newtheorem{claim}{Claim}[section]

\newtheorem*{theoremA}{Theorem A}
\newtheorem*{theoremB}{Theorem B}
\newtheorem*{theoremC}{Theorem C}

\newtheorem*{theoremM}{McMullen's Theorem}
\newtheorem*{lemmaA}{Surgery Principle}
\newtheorem*{conj}{ Goldberg-Milnor Conjecture}
\newcommand{\omC}{\widehat{\mathbb{C}}}

\newcommand{\sm}{\setminus}
\newcommand{\wh}{\widehat}

\newcommand{\g}{\gamma}

\begin{document}
	\title{On the parabolic Fatou domains
	}
	\author{Ning Gao
		\and Yan Gao\thanks{The second author is supported by the National Key R\&D Program of China Grant Nos. 2021YFA1003203, the NSFC Grant Nos. 12131016 and
			12322104, and the NSFGD Grant Nos. 2023A1515010058.}
		\and Wenjuan Peng\thanks{The third author is supported by the NSFC Grant Nos. 12122117, 12271115 and 12288201.}}
	\date{\today}
	\maketitle
\begin{abstract}
Let $f$ be a rational map with an infinitely-connected fixed parabolic Fatou domain $U$.  We prove that there exists a rational map $g$ with a completely invariant parabolic Fatou domain $V$, such that $(f,U)$ and $(g,V)$ are conformally conjugate, and each non-singleton Julia component of $g$ is a  Jordan curve which bounds a  superattracting Fatou domain of $g$ containing  at most one postcritical point.  Furthermore, we show that if the Julia set of  $f$ is a Cantor set,  then the parabolic Fatou domain can be perturbed into an attracting one without affecting the topology of the Julia set.
\end{abstract}
	
\section{Introduction}
Let $f:\widehat{\mathbb{C}}\rightarrow \widehat{\mathbb{C}}$ be a rational map of degree $d\geq 2$. The {\bf Fatou set} $F(f)$ is  the set of points $z\in\widehat{\mathbb{C}}$ such that the iteration sequence $\{f^n\}_{n =1}^{\infty}$ forms a normal family in a neighborhood of $z$. The complement of the Fatou set is the  {\bf Julia set} $J(f)$.  By definition, the Fatou set is open while the Julia set is closed. Moreover, they are both completely invariant under the iteration of $f$, i.e., $f^{-1}(F(f))=F(f)$ and $f^{-1}(J(f))=J(f)$. Refer to \cite{M} for more properties of Fatou and Julia sets. The {\bf postcritical set} of $f$ is defined as
\[P(f)=\overline{\{f^n(c)\ | \ f'(c)=0, n\geq1\}}.\]
	
A connected component of $F(f)$ is called a {\bf Fatou domain}. Since $F(f)$ is open and completely invariant, the map $f$ sends a Fatou domain onto a Fatou domain. Therefore, a Fatou domain $U$ is either {\bf preperiodic}, i.e., $f^m(U)=f^{\ell+m}(U)$ for some $\ell\geq0,m\geq 1$; or {\bf wandering}, otherwise. Furthermore, a preperiodic domain is called {\bf periodic} if $\ell=0$.
	
By the works of Fatou, Siegel, Arnold and Herman (see \cite{F,S,H}), there are four types of periodic Fatou domains: \textbf{(super)attracting, parabolic, Siegel disk} and \textbf{Herman ring}. A fundamental result  in complex dynamics, due to Sullivan, asserts that rational maps have no wandering Fatou domains. Then the classification of Fatou domains is complete.
	
Suppose that $R_1$ and $R_2$ are two rational maps, and $D_1$ and $D_2$  are two sets in $\omC$ with $R_1(D_1)=D_1$ and  $R_2(D_2)=D_2$. We say $(R_1, D_1)$ and $(R_2, D_2)$ are topologically (quasiconformally or holomorphically) conjugate if there exists a topological (quasiconformal or conformal) map $\phi$ from $D_1$ onto $D_2$ such that $\phi\circ R_1=R_2\circ \phi$ on $D_1$. We call the map $\phi$ a topological (quasiconformal or conformal) conjugation.
	
A natural follow-up question would be to find holomorphic models for each type of periodic Fatou domain. Roughly speaking, for a rational map $f$ with a fixed Fatou domain $U$, we call $g$  a {\bf (holomorphic) model} of $(f,U)$ if $g$ is a  ``canonical'' rational map satisfying the following two properties:
\begin{itemize}
\item $g$ has a completely invariant Fatou domain $U_g$ such that $(f,U)$ and  $(g,U_g)$ are conformally conjugate ;
\item $g$ is unique up to a conformal conjugation, i.e., if $h$ is another canonical rational map satisfying the above property, then $g$ and $h$ are conformally conjugate on $\widehat{\mathbb{C}}$.
\end{itemize}

It is known that if $U$ is a Siegel disk or Herman ring of $f$, the model for $(f, U)$ is an irrational rotation. A fixed attracting or parabolic Fatou domain $U$ is either simply or infinitely-connected, and when $U$ is a simply-connected, a suitable Blaschke product serves as the model for $(f, U)$.	

 In \cite[Theorem 1.1]{CP}, Cui and Peng established the model for infinitely-connected attracting Fatou domains. The canonical maps adapted there are \textbf{simple attracting maps}: any such map $g$ has a completely invariant attracting Fatou domain $U_g$, where every non-singleton component of $\partial U_g$ is a quasi-circle. This quasi-circle bounds an eventually superattracting Fatou domain that contains at most one postcritical point.
\begin{theoremA}[Cui-Peng]
Let $U$ be an infinitely-connected fixed attracting Fatou domain of a rational map $f$. Then there exists a simple attracting map  $g$ as a  model of $(f,U)$, i.e.,
\begin{itemize}
\item [(1)] $(f,U)$ and  $(g,U_g)$ are conformally conjugate;
\item [(2)] such $g$ is unique up to a conformal conjugation.
\end{itemize}
\end{theoremA}
	
Thus, for the holomorphic model problem, only the case of infinitely-connected parabolic Fatou domains remains unresolved. Motivated by Cui-Peng's work, we introduce simple parabolic maps.

A rational map $g$ is called a \textbf{simple parabolic map} if it has a completely invariant parabolic Fatou domain $U_g$, where each non-singleton component $J$ of $\partial U_g$ is a Jordan curve. This Jordan curve bounds an eventually superattracting Fatou domain that contains at most one postcritical point. Furthermore, if the forward orbit of $J$ avoids the unique parabolic fixed point of $g$, then $J$ is a quasi-circle.

Our first result shows that simple parabolic maps serve as candidates for the holomorphic model of infinitely-connected parabolic Fatou domains.
\begin{theorem}\label{thm:1.1}
Let $f$ be a rational map with an infinitely-connected fixed parabolic Fatou domain of $U$. Then there exists a simple parabolic  map $g$ such that  $(f,U)$ and $(g,U_g)$ are conformally conjugate.
\end{theorem}
	
In the proof of Theorem A, the authors constructed the simple attracting map $g$ directly by using attracting puzzle pieces from $U$ and the quasiconformal surgery. The construction hinges on disjoint boundaries for puzzle pieces with different depths, which enables straightforward quasiconformal surgery on inter-puzzle annuli.
	
In the parabolic case, puzzle pieces from $U$ still exist, but the boundaries of puzzle pieces at different depths may intersect at iterated preimages (on $\partial U$) of the parabolic fixed points. Such intersections present an obstruction to quasiconformal surgery. Our idea for constructing a simple parabolic map from $U$ is to make use of simple attracting maps.
	
More precisely, we first perform \textbf{plumbing surgery} on $U$ and combine it with quasiconformal surgery to obtain a sequence of simple attracting maps; and then show that this sequence converges to the desired simple parabolic map.
	
The  plumbing surgery was proposed by Cui and Tan \cite{CT} to study the hyperbolic-parabolic deformation of rational maps. The method that combines plumbing surgery and quasiconformal surgery  was first applied by Peng, Yin and Zhai to prove the density of hyperbolicity for rational maps with Cantor Julia sets \cite{PYZ}.
	
At present, we cannot assert that simple parabolic maps serve as models for infinitely-connected parabolic Fatou domains, as their uniqueness remains unproven. This uniqueness issue essentially relies on rigidity results for simple parabolic maps, which are currently unknown. Note that any rational map with a Cantor Julia set and a parabolic fixed point is clearly a simple parabolic map. Rigidity theorems for such maps were established by Yin and Zhai \cite{Z1,Z2}.

The approach we use to prove Theorem \ref{thm:1.1} naturally extends to address another problem proposed by Goldberg and Milnor \cite{GM}.
	
\begin{conj}
For any polynomial $P$ having a parabolic cycle, the immediate basin of the parabolic cycle can be converted to be attracting by a small perturbation, and the perturbed polynomial on its Julia set is topologically conjugate to the original polynomial $P$ on $J(P)$.
\end{conj}

This conjecture was addressed in the setting of geometrically finite polynomial maps with connected Julia sets by Haissinssky \cite{Hai}.
For a geometrically finite rational map, Cui and Tan \cite{CT}, and Kawahira \cite{Ka1, Ka2} gave an affirmative answer to this conjecture, respectively, using different approaches. 
In our work, we consider this perturbation problem for simple parabolic maps.

\begin{theorem}\label{thm:middle}
Let $f$ be a simple parabolic map. Then there exists a simple attracting map $g$ such that $(f,J(f))$ and $(g,J(g))$ are topologically conjugate.
\end{theorem}
	
In fact, we can construct a sequence $\{g_n\}$ of simple attracting maps satisfying the conclusion of Theorem \ref{thm:middle}, which converges to a simple parabolic map $f_*$ such that $(f, U_f)$ is conformally conjugate to $(f_*, U_{f_*})$. However, due to the absence of the uniqueness result of simple parabolic maps, we cannot conclude that $f$ has the stable perturbation in the sense of Goldberg-Milnor.
	
Using the rigidity results of Yin and Zhai, we can prove the Goldberg-Milnor conjecture for a special class of simple parabolic maps.
	
\begin{theorem}\label{thm:1.2}
The Goldberg-Milnor conjecture holds for any rational map with a Cantor Julia set and a parabolic fixed point.
\end{theorem}

The paper is organized as follows. In Section 2, we study the topology of boundary components for an infintely-connected periodic Fatou domain. In Section 3, we construct a double-subscript sequence $\{f_{n,t}\ | \ n\geq1,t\in(0,1)\}$ based on $(f,U)$, which is a foundation for the proofs of Theorems 1.1--1.3. In Section 4, we prove Theorem \ref{thm:1.1}, and in Section 5, we prove Theorems \ref{thm:middle} and \ref{thm:1.2}.
	
Throughout this paper, a {\bf disk} means a Jordan domain in $\mathbb{C}$, and a {\bf closed disk} means the closure of a disk.  For simplicity, we use the term \textbf{component} to refer to a connected component.

\section{Boundary components of Fatou domains}\label{sec:2.1} 
For a non-empty connected and compact set $E\subset \mathbb{C}$, its {\bf filling} $\widehat{E}$  is defined as the union of $E$ and all bounded components of $\mathbb{C}\sm E$; and it is  called {\bf full} if $E=\widehat{E}$ is connected. \vskip 0.1cm

Let $R$ be any rational map and $W$ be a fixed infinitely-connected Fatou domain of $R$. Then $W$ is attracting or parabolic. Let  $ {\mathcal E }_R={\cal E}_R(W)$ denote the collection of all components of $\partial W$. For any $E\in\mathcal E_R$, we have
\begin{itemize}
\item [(1)]  $\widehat{E}$ is a component of $\mathbb{C}\setminus W$;\vspace{2pt}
\item [(2)] $R(E)\in\mathcal E_R$; and
\item [(3)]  $R(\widehat{E})=\widehat {R(E)}$ if $\widehat{E}\cap R^{-1}(W)=\emptyset$, and $R(\widehat{E})=\widehat{\mathbb{C}}$ otherwise.
\end{itemize}
By the above statement (2) and the fact that $R(W)=W$, we obtain a surjective map 
\[\sigma_R:\mathcal E_R\to \mathcal E_R,\ E\mapsto R(E).\]
Using the forward iteration of  $\sigma_R$, we can define the  {(pre)periodic} or {wandering} elements of $\mathcal{E}_R$, and the orbits of elements of $\mathcal E_R$.\vskip 0.1cm
	
For any $E\in\mathcal{E}_R$, we define the {\rm degree} of $\sigma_R$ on $E$ as follows. Choose a disk $D\supset \widehat{\sigma_R(E)}$ such that the annulus $A:=D\setminus \widehat{\sigma_R(E)}$ does not contain the critical values of $R$. Let $D_E$ be the component of $R^{-1}(D)$ containing $E$. Since $A$ is disjoint from the critical values, the set $A_E:=D_E\setminus \widehat{E}$ is still an annulus, with one boundary component $E$, such that $R:A_E\to A$ is a covering. Then we define
\[{\rm deg}_E\sigma_R={\rm deg}(R:A_E\to A).\]
It is easy to check that this definition is independent on the choice of $D$, and it holds that ${\rm deg}_E \sigma_R^2={\rm deg}_E\sigma_R\cdot{\rm deg}_{\sigma_R(E)}\sigma_R$. If ${\rm deg}_E\sigma_R>1$, we call $E$ a {\bf critical} element of $\mathcal E_R$.
	
\begin{lemma}\label{lem:finite}
If $E$ is an critical element of $\mathcal E_R$, then $\widehat{E}$ contains critical points of $R$. Consequently, there are finitely many critical components in $\mathcal E_R$.
\end{lemma}
\begin{proof}
Choose a disk $D\supset \widehat{\sigma_R(E)}$ such that the annulus $D\setminus \widehat{\sigma_R(E)}$ avoids the critical values of $R$. Let $D_E$ and $A_E$ be defined as in the definition of ${\rm deg}_E\sigma_R$. Since ${\rm deg}_E\sigma_R>1$, by definition, it follows that $\deg(R:D_E\to D)$ is larger than one. Then $D_E$ contains critical points of $R$ by the Riemann-Hurwitz formula. Thus we have $$\emptyset\not=D_E\cap C(R)=(D_E\setminus A_E)\cap C(R)\subset \widehat{E}\cap C(R),$$ where $C(R)$ denotes the set of critical points of $R$. Then the lemma is proved.
\end{proof} 
	
If the Fatou domain $W$ is completely invariant, then every element of $\mathcal{E}_R$ is a Julia component. In this case, the following theorem by McMullen (\cite[Theorem 3.4]{Mc}) can be applied.

\begin{theoremM}
Let $E$ be a non-singleton Julia component  of a rational map $R$ of degree $d\geq 2$ such that $R(E) = E$. Then there exist a rational map $h$ of degree at least $2$ and a quasiconformal map $\varphi: \widehat{\mathbb{C}} \to \widehat{\mathbb{C}}$ such that $\varphi(E) = J(h)$ and $\varphi \circ R = h \circ \varphi$ on $E$.
\end{theoremM}	

\begin{lemma}\label{lem:component}
Suppose that $R^{-1}(W)=W$, and let $E$ be an element of $\mathcal E_R$.
\begin{itemize}
\item [(1)] The set $E$ is a point if and only if its orbit contains no periodic critical components.
\item [(2)] Suppose that $R(E)=E$, ${\rm int}\widehat{E}$ contains exactly one critical point, which is fixed by $R$, and $E\cap C(R)=\emptyset$. Then $E$ is a Jordan curve. Furthermore, if $E$ contains no parabolic points, then $E$ is a quasi-circle.
\end{itemize}
\end{lemma}
\begin{proof}
(1) The necessity is obvious. So it is enough to prove the sufficiency.
		
It is proved in \cite{Z3} that $J(R)$ is a Cantor set if and only if each critical element of $\mathcal E_R$ is not periodic. The proof of this result implies that each wandering element of $\mathcal{E}_R$ is a singleton. Then we only need to consider the preperiodic case.
		
In this case, it is enough to show that if $\sigma_R(E)=E$ and ${\rm deg}_E\sigma_R=1$, then $E$ is a singleton. Suppose on the contrary that $E$ is not a singleton. Since $E$ is a Julia component, we can apply McMullen's Theorem to $(R, E)$, and then obtain that ${\rm deg} (R|_E)>1$. Note that for a Julia component $E$, it holds that ${\rm deg} (R|_E)={\rm deg}_E\sigma_R$. This contradicts   ${\rm deg}_E\sigma_R=1$. \vspace{3pt}
		
(2) Since $W$ is completely invariant, we have $R(\widehat{E})=\widehat{E}$. Then ${\rm int}\widehat{E}\subset F(R)$. From the properties on $\widehat{E}$, we conclude that ${\rm int}\widehat{E}$ is a simply-connected superattracting Fatou domain of $R$, denoted by $\Omega_1$. Note that $\Omega_2=\omC\sm\overline{\Omega}_1$ is also a simply-connected domain, and  $E = \partial\Omega_1 = \partial\Omega_2$. \vskip 0.1cm
		
Set $\deg(R|_{E})=d_0\geq 2$.  By applying McMullen's Theorem to $(R,E)$, we obtain a rational map $h$ of degree $d_0$ and a quasiconformal map $\varphi\colon\widehat{\mathbb{C}}\to\widehat{\mathbb{C}}$ such that $\varphi(E) = J(h)$ and $\varphi \circ R = h \circ \varphi$ on $E$. Then $\varphi(\Omega_1)$ and $\varphi(\Omega_2)$ are the only two Fatou components of $h$, and they are both completely invariant. 
		
As $E$ contains no critical points of $R$, $J(h) = \varphi(E)$ contains no critical points of $h$. Thus $h$ is geometrically finite. So $J(h)$ is locally connected \cite{TY}. Since $J(h) $ is the common boundary of two Fatou domains, it is a Jordan curve, implying that $E$ is also a Jordan curve.
		
Furthermore, if $E$ contains no parabolic points of $R$, we have that $J(h) =\varphi(E)$ contains no parabolic points of $h$ as $\varphi$ is quasiconformal. Thus $h$ is a hyperbolic map. 
Since $h$ is expanding on $J(h)$, there exists a holomorphic covering map $H\colon A_2 \to A_1$, where $A_1$ and $A_2$ are annuli satisfying that $A_2 \subset A_1$ and $A_2$ seperates the boundary components of $A_1$. Then $\varphi(E)=\bigcap_{n \geq 0} H^{-n}(A_1)$ is a quasi-circle (see \cite[Lemma 3.4]{CP}). Consequently, $E$ is a quasi-circle.
\end{proof}
	
\section{Fundamental sequences }
Let $f$ be a rational map with a fixed  infinitely-connected  parabolic Fatou domain $U$, and $d:={\rm deg}(f|_U)$.
	
In this section,  making using of plumbing surgery and quasiconformal surgery to $(f,U)$, we can construct a double-subscript sequence $\{f_{n,t}\ | \ n\geq1,t\in(0,1)\}$ of degree-$d$ rational maps  with completely invariant attracting Fatou domains.
	
These sequences are foundations in our proofs of Theorems 1.1--1.3. Roughly speaking, fixing any $t\in(0,1)$ and letting $n\to\infty$, we obtain a simple attracting map $f_t$ required in Theorems 1.2 and 1.3;   letting $t\to 0$, $f_t$ converges to a simple parabolic map required in Theorem 1.1.
	
\subsection{Parabolic puzzles}\label{sec:puzzle}
The construction of parabolic puzzles is similar to that of attracting ones given in \cite[Section 2]{CP}, with the substitution of an attracting petal of the parabolic fixed point for a linearization domain of the attracting fixed point.
	
Without loss of generality, we may assume that $\infty\in U, f(0)=0$ and $U$ is the immediate parabolic basin of $0$. By the Leau-Fatou Flower Theorem \cite{M}, there exists a disk $U_0\subset U$ with smooth boundary except at $0$, called an \textbf{attracting petal} of $0$, such that
\begin{itemize}[leftmargin=1cm]
\item[(1)] $0\in\partial U_0,  f(\overline{U_0})\subset U_0\cup\left\{0\right\}$.
		
\item[(2)] $f:U_0\rightarrow f(U_0)$ is conformal, and $(\partial U_0\sm \left\{0\right\})\cap P(f)=\emptyset$.
		
\item[(3)] $\{f^n|_{U_0}\}$  converges locally and uniformly to $0$, as $n\rightarrow\infty$.
		
\item[(4)] for any $z\in U$, there exists an integer $k\geq 1$ such that $f^{k}(z)\in U_0$.
\end{itemize}

Denote by $\left\langle f\right\rangle $ the grand orbit of $f$. Then $U_0 / \left\langle f\right\rangle $ is conformally isomorphic to the infinite cylinder $\mathbb{C}/\mathbb{Z}$, which is called an \textbf{attracting cylinder}.  Let $\pi$ denote the natural projection from an attracting petal to the attracting cylinder. An attracting petal  is called \textbf{regular} if the arc $\pi(\partial U_0\sm \{0\})$ lands on punctures at both ends. Every attracting petal contains a regular attracting petal (refer to \cite[Proposition 2.15]{CT}). So we always assume an attracting petal is regular in this paper.
	
For each $n\geq1$, let $U_n$  denote the component of $f^{-n}(U_0)$ containing $U_0$. It follows that $U_{n}\subset U_{n+1}$, $U=\bigcup_{n\geq 0}U_n$, $\partial U_{n}\cap\partial U_{n+1}\subset f^{-n}(0)$, and $f:U_{n+1}\rightarrow U_{n}$ is a holomorphic proper map. There exists an integer $N\geq 1$ such that for all $n\geq N$,  $\deg(f:U_n\rightarrow U_{n-1})=d$. Then $U_{N}$ contains all critical points of $f$ in $U$.
	
For each $n\geq0$, set $Z_n=f^{-n}(0)\cap\partial U_n$. It follows that  $Z_n\subset Z_{n+1}$ for $n\geq 0$ and $Z_{n}=f^{-1}(Z_{n-1})\cap \partial U_n$ for $n\geq N$.
	
For each $n\geq 0$, let $\mathcal{P}_n$ denote the collection of all components of $\mathbb{C}\setminus U_{N+n}$, which are called  \textbf{(parabolic) puzzle pieces} of depth $n$ (see Figure \ref{fig:1}). We remark that the puzzle pieces in existing literature are open sets, while here we take them closed for technical reasons.	  
\begin{figure}[htbp]
	\centering
	\includegraphics[width=14cm]{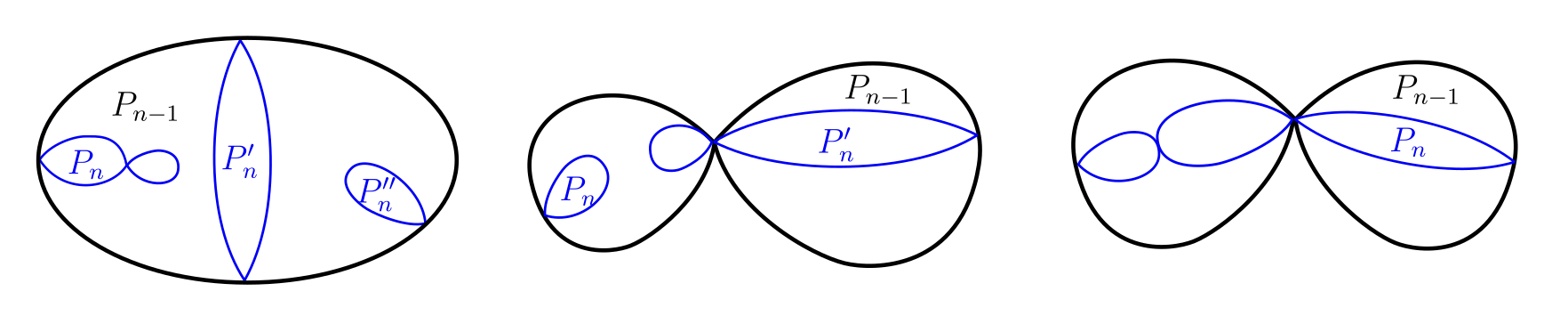}
	\caption{Puzzle pieces of $f$. $P_{n-1}$ denotes the puzzle piece of depth $n-1$, and $P_{n}$, $P_{n}^{\prime}$, $P_{n}^{\prime\prime}$ denote the puzzle pieces of depth $n$ contained in $P_{n-1}$.}
	\label{fig:1}
\end{figure}

\begin{lemma}\label{puzzle} The puzzle pieces satisfy the following properties.\vspace{3pt}
		
$\mathrm{(P1)}$ Fix $P_n\in\mathcal{P}_n$ for $n\geq 0$. Then the following statements hold.
\begin{itemize}[leftmargin=1.5cm]
\item[(a)] $P_n$ is  full  and ${\rm int}P_n$ has finitely many components, each of which is a disk.
\item[(b)] If the boundaries of two components of ${\rm int}P_n$ intersect, then the intersection is contained in  $\big(\bigcup_{i\geq 0}f^{-i}(C(f))\big)\bigcap f^{-(N+n)}(0)$.
\item[(c)] $\partial P_n\cap (\widehat{\mathbb{C}}\sm U)\subset Z_{N+n}$.
\item[(d)] $P_n$ is disjoint from any other puzzle piece in $\mathcal{P}_n$.
\item[(e)] For any $P_{n+1}\in\mathcal{P}_{n+1}$, if $P_{n+1}\cap P_{n}\neq \emptyset$, then $P_{n+1}\subset P_{n}$ and $\partial P_{n+1}\cap \partial P_{n}\subset Z_{N+n}$; on the other hand,  any $z\in Z_{N+n}\cap\partial P_n$ belongs to  a unique puzzle piece in $\mathcal{P}_{n+1}$.
\end{itemize} \vskip 0.1cm
		
$\mathrm{(P2)}$  For any $n\geq 0$, $$\bigcup_{P_n\in\mathcal{P}_n}P_n\supset \omC\sm U \quad \text{and}\quad \bigcap_{n\geq 0}\bigcup_{P_n\in\mathcal{P}_n}P_n=\widehat{\mathbb{C}}\sm U.$$
		
$\mathrm{(P3)}$ For each $E\in{\mathcal E }_f$ and $n\geq 0$, there is a unique puzzle piece $P_n(E)\in\mathcal{P}_n$ containing $E$, and it holds that  $P_{n+1}(E)\subset P_n(E)$ and $\bigcap_{n\geq 0}P_n(E)=\widehat{E}$.
\end{lemma}
	
\begin{proof}
$\mathrm{(P1)}$ (a)-(d) follow directly from the construction of puzzle pieces.
		
Note that \[\bigcup_{P_{n+1}\in\mathcal{P}_{n+1}}P_{n+1}= \omC\sm U_{N+n+1}\subset \omC\sm U_{N+n}=\bigcup_{P_n\in\mathcal{P}_n}P_n.\]
So if $P_{n+1}\cap P_{n}\neq \emptyset$, then $P_{n+1}\subset P_{n}$.  Since \[\partial P_{n+1}\subset\partial U_{N+n+1},\ \ \partial P_{n}\subset\partial U_{N+n}\quad \text{and}\quad \partial U_{N+n+1}\cap\partial U_{N+n}\subset f^{-(N+n)}(0), \] we have $\partial P_{n+1}\cap \partial P_{n}\subset Z_{N+n}$. For any $z\in Z_{N+n}\cap\partial P_n$, we have $f^{N+n}(z)=0$. Assume by contradiction that $z\notin\mathcal{P}_{n+1}$. Then $z\in  U_{N+n+1}$, hence $f^{N+n+1}(z)\in U_0$. This contradicts  $f^{N+n+1}(z)=0$.  This proves (e). \vskip 0.1cm
		
$\mathrm{(P2)}$ Since $U_{N+n}\subset U$,  $\omC\sm U\subset\omC\sm U_{N+n}$, which is equal to $\bigcup_{P_n\in\mathcal{P}_n}P_n$. \vskip 0.1cm
From $U=\cup_{n\geq 0}U_n=\cup_{n\geq 0}U_{N+n}$, we obtain	\[\omC\sm U=\omC\sm \bigcup_{n\geq 0}U_{N+n}= \bigcap_{n\geq 0}\big(\omC\sm U_{N+n}\big)=\bigcap_{n\geq 0}\bigcup_{P_n\in\mathcal{P}_n}P_n.\]
		
$\mathrm{(P3)}$ By (e) of (P1), we know that $P_{n+1}(E)\subset P_{n}(E)$ for any $n\geq0$. For any $z\in\bigcap_{n\geq 0}P_n(E)$, we have $z\in\widehat{\mathbb{C}}\sm U$. Otherwise, $z\in U_{N+m}$ for some $m\geq 0$, which implies that  $z\notin P_m(E)$, a contradiction. Thus $\bigcap_{n\geq 0}P_n(E)\subset \widehat{\mathbb{C}}\sm U$. Note that $\bigcap_{n\geq 0}P_n(E)$ is connected and $\partial P_n(E)\sm f^{-(N+n)}(0)\subset U$.  It follows that  $\bigcap_{n\geq 0}P_n(E)=\widehat{E}$.
\end{proof}
	
\subsection{Construction of fundamental sequences}
The construction of sequences $\{f_{n,t}\ | \ n\geq 0,t\in(0,1)\}$ from $(f,U)$ involves two steps.
	
\subsubsection{The plumbing surgery.}\label{sec:plumbing}
\quad\ \   Following  \cite[Section 2]{CT}, there are two disjoint disks $S_\pm$  with smooth boundaries except at $0$, called  {\bf sepals} of the parabolic point $0$, such that
\begin{itemize}
\item $\overline{S_+}\cup \overline{S_-}\subset U\cup\{0\}$ and $\overline{S_+}\cap\overline{S_-}=\{0\}$,\vskip 0.1cm
\item both $S_\pm$ intersect the attracting petal $U_0$ and are disjoint from $P(f)$,\vskip 0.1cm
\item $f:\overline{S_+}\to \overline{S_+}$ and $f:\overline{S_-}\to\overline{S_-}$ are both homeomorphisms.
\end{itemize}
	
Fix $\delta\in\{+,-\}$. The quotient space $S_\delta/\left\langle f\right\rangle $ is an once-punctured disk. Then there is a natural holomorphic projection $\pi_\delta: S_\delta\rightarrow \mathbb{D}^{*}$, where $\mathbb{D}^{*}=\{z\in\omC \, |\, 0<|z|<1\}$, such that $\pi_\delta(z_1)=\pi_\delta(z_2)$ if and only if $f^k(z_1)=z_2$ for some $k\geq0$. Clearly, this map is a universal  covering. For any $0<t<1$, set  $S_\delta(t):=\pi_\delta^{-1} (\mathbb{D}^{*}(t))$ where $\mathbb{D}^{*}(t)=\{z\in\omC \, |\, 0<|z|<t\}$, and  $L_\delta(t)=\partial S_\delta (t)\sm\{0\}$. By definition of $\pi_\delta$, we have 
\begin{equation}\label{eq:11}
f(L_\delta(t))=L_\delta(t)\ \text{ and }\ f(S_\delta(t))=S_\delta(t).
\end{equation}
	
Fix any $t\in(0,1)$, we denote $S_0(t):=S_+(t)\cup S_-(t)$ and $S_0:=S_+\cup S_-$. There is a conformal map $\tau_0:S_0\sm \overline{S_0(t^2)}\rightarrow S_0\sm \overline{S_0(t^2)}$ such that
\begin{itemize}
\item[(1)] for any $s\in(t^2,1)$, $\tau_0(L_+(s))=L_-(t^2/s)$, and \vskip 0.1cm
\item[(2)] $\tau_0^2=id$ and $f\circ\tau_0=\tau_0\circ f$.
\end{itemize}
	
Define an equivalence relation in $\omC\sm\overline{S_0(t^2)}$ by $z_1\sim z_2$ if $z_1= z_2$ or $\tau_0(z_1)=z_2$. The quotient space $\big(\omC\sm\overline{S_0(t^2)}\big)/\sim$ is  holomorphically isomorphic to a two-punctured sphere. Let $\pi_0:\omC\sm\overline{S_0(t^2)}\rightarrow \omC\setminus \{x_0,y_0\}$ be the projection such that $\pi_{0}(z_1)=\pi_{0}(z_2)$ if and only if $z_1\sim z_2$. It then holds that
\begin{itemize}
\item $\pi_0(U_0\setminus \overline{S_0(t^2)})$ is an one-punctured disk with smooth boundary, denoted by $V_0^*$,
\item $\pi_0$ is univalent on both $\omC\setminus \overline{S_0(t)}$ and $S_\delta\setminus \overline{S_\delta(t^2)}$, with $\delta\in\{+,-\}$.
\end{itemize}
	
Set $S_n(s)=f^{-n}\big(S_0(s)\big)$ for $0<s\leq 1$ and $n\geq 1$, where $S_0(1)=S_0$. The map $\tau_0$ can be lifted to $\tau_n:S_n(1)\sm\overline{S_n(t^2)}\rightarrow S_n(1)\sm\overline{S_n(t^2)}$ through $f^n$ for each $n\geq 1$ as follows.
	
Fix $n\geq 1$. For any $z\in Z_n$, it belongs to a boundary component $B$ of $U_n$, and denote by $m_z$ the number of components of $\wh{B}\setminus\{z\}$. Take a small disk-neighborhood $D$ of $z$ such that $D\setminus \wh{B}$ has $m_z$ components. For any such component  $W$, there are exactly two components of $S_n(s)$ intersecting $W$ and having the common boundary point $z$. Denote their union by $S_{W}(s)$. Then $f^n: S_{W}(s)\rightarrow S_0(s)$ is conformal. There is a conformal map $$\tau_{W}:S_{W}(1)\sm \overline{S_{W}(t^2)}\rightarrow S_{W}(1)\sm \overline{S_{W}(t^2)}$$ such that $f^n\circ\tau_{W}=\tau_0\circ f^n$. Since $S_n(1)\sm\overline{S_n(t^2)}$ is the union of all such $S_W(1)\sm \overline{S_W(t^2)}$, we have a map  \[\tau_n:S_n(1)\sm \overline{S_n(t^2)}\rightarrow S_n(1)\sm \overline{S_n(t^2)}\] defined as $\tau_n=\tau_W$ in $S_W(1)\sm \overline{S_W(t^2)}$, see Figure \ref{fig:2}. \begin{figure}[htbp]
	\centering
	\includegraphics[width=12cm]{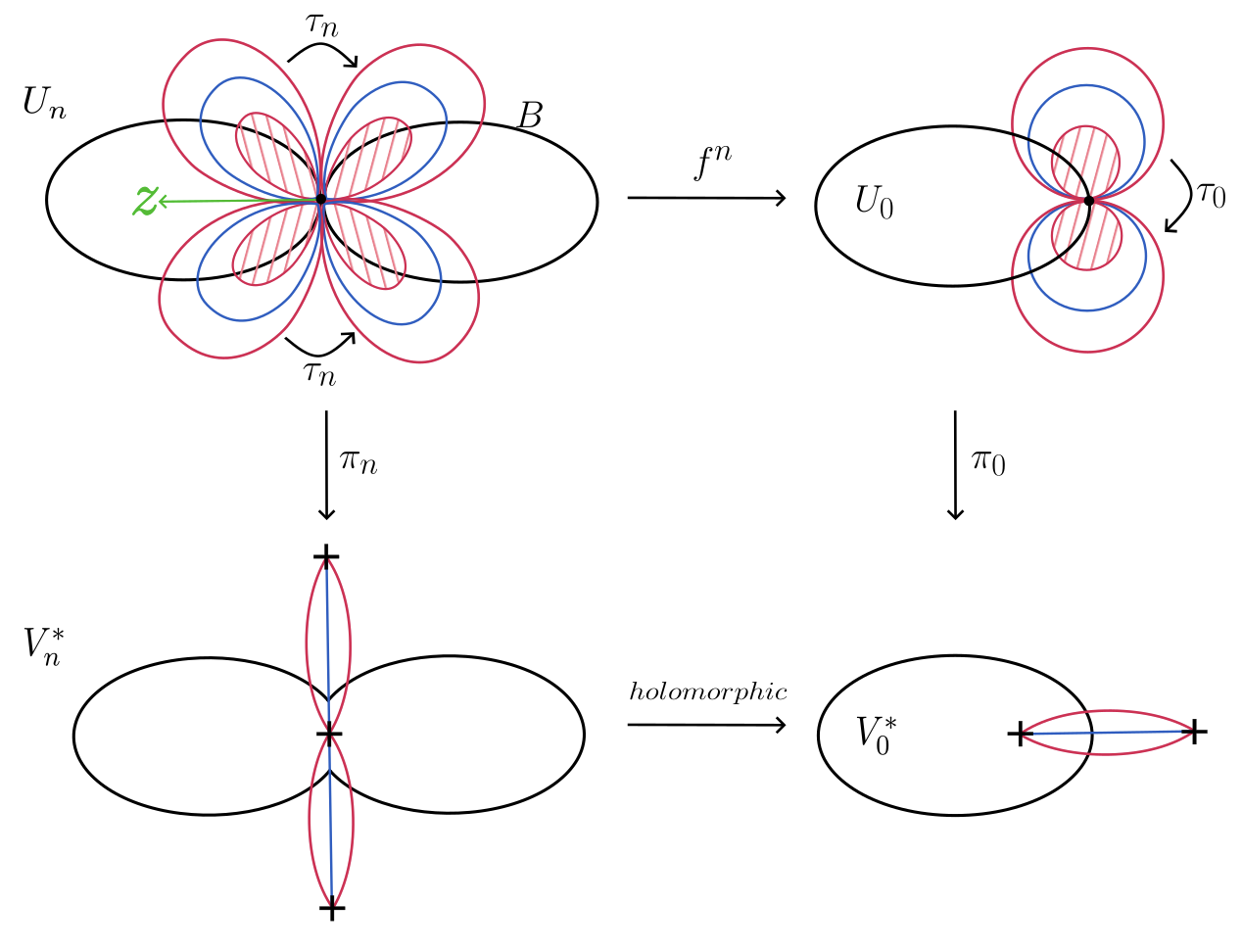}
	\caption{Surgery at the point $z\in Z_n $. }
	\label{fig:2}
\end{figure}
	
Define an equivalence relation in $\widehat{\mathbb{C}}\sm\overline{S_n(t^2)}$ by $z_1\sim z_2$ if $z_1= z_2$ or $\tau_n(z_1)=z_2$.  Then the quotient space $\big(\omC\sm\overline{S_n(t^2)}\big)/\sim$ is holomorphically isomorphic to a  punctured sphere with finitely many punctures. Thus there exist a finite set $X_{n}\subset \widehat{\mathbb{C}}$ and a holomorphic surjective map  \[\pi_{n}:\widehat{\mathbb{C}}\sm\overline{S_n(t^2)}\rightarrow \widehat{\mathbb{C}}\sm X_{n}\]  such that $\pi_{n}(z_1)=\pi_{n}(z_2)$ if and only if $z_1\sim z_2$. There are two special punctured points $x_n,y_n\in X_n$, which correspond to the parabolic fixed point $0$. Moreover,
\[\text{$\pi_n$ is univalent on both $\omC\sm \overline{S_n(t)}$ and each component of $S_n(1)\sm\overline{S_n(t^2)}$.} \eqno(\ast) \label{deg1} \]
	
Set $$V_{n}^*=\pi_{n}\big(U_n\sm\overline{S_n(t^2)}\big)\quad \text{and} \quad \widetilde{V}_{n-1}^*=\pi_{n}\big(U_{n-1}\sm\overline{S_{n-1}(t^2)}\big).$$
Then they both have punctures in $X_n$. Denote by $V_n$ ({\it resp}. $\widetilde{V}_{n-1}$) the union of $V_n^*$ ({\it resp}. $\widetilde{V}_{n-1}^*$) and its punctures. It follows from Lemma \ref{puzzle} that
\begin{itemize}
\item any component of $\partial V_n$ and $\partial \widetilde{V}_{n-1}$ is a smooth Jordan curve, and
\item $\widetilde{V}_{n-1}\Subset V_n$, and one of the two special punctured points $x_n,y_n$, say $x_n$, belongs to $ V_n$.
\end{itemize}
	
For any component $P_n$  of $\mathbb{C}\setminus U_n$, and $P_{n-1}$ of $\mathbb{C}\setminus U_{n-1}$, $$\pi_n\big(P_n\setminus\overline{S_n(t^2)}\big)\quad \text{and} \quad \pi_{n}\big(P_{n-1}\sm\overline{S_{n-1}(t^2)}\big)$$
are closed disks with punctures in $X_n$, and their closures $B_n$ and $\widetilde{B}_{n-1}$ are complementary components of $V_n$ and $\widetilde{V}_{n-1}$, respectively.
	
There exists a holomorphic proper map $G_{n}: V_{n}\to \widetilde{V}_{n-1}$ of degree $d$, such that
	\[\text{$G_{n}\circ \pi_{n}=\pi_{n}\circ f$ in $U_n\sm\overline{S_n(t^2)}$\ \ and \ $G_n(x_n)=x_n$ (see Figure \ref{fig:3}).}\]
 \begin{figure}[htbp]
 	\centering
 	\includegraphics[width=13cm]{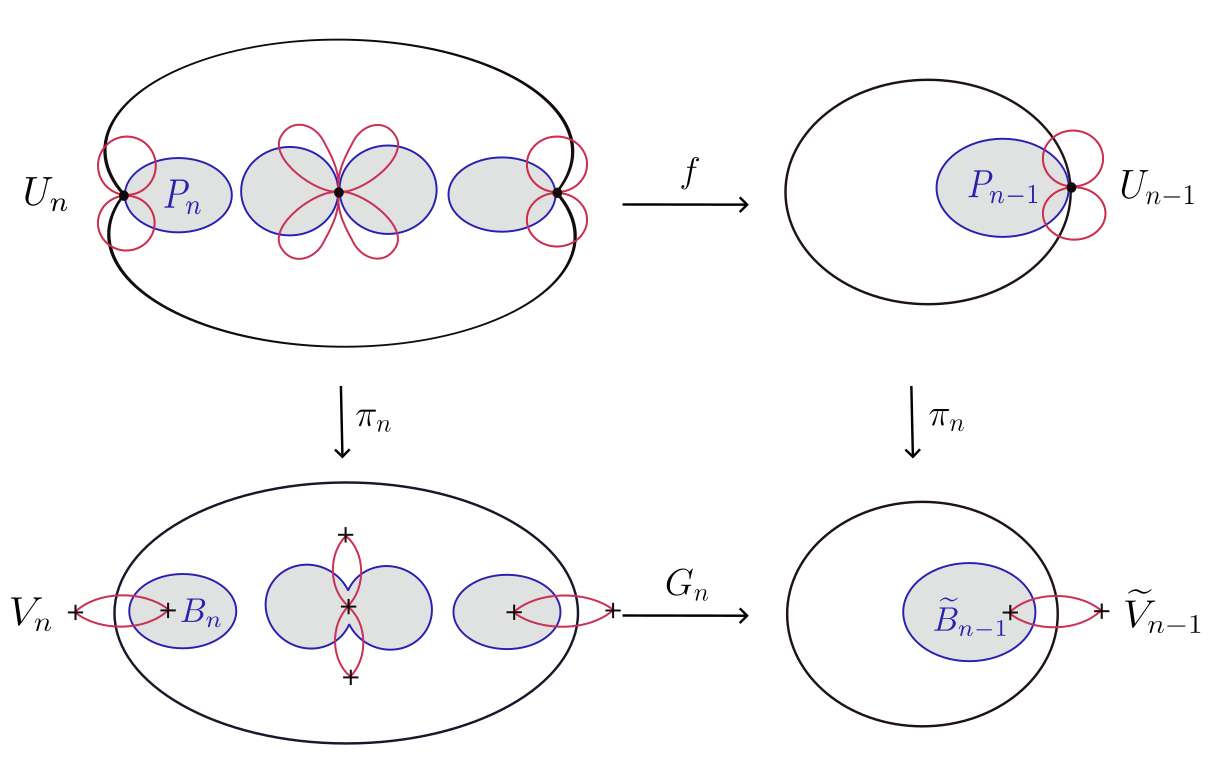}
 	\caption{ The induced map $G_{n}$ after surgery.  $U_n$ ($U_{n-1}$) and $V_{n}$ ($\widetilde{V}_{n-1}$) denote the complements of the domains marked in gray.}
 	\label{fig:3}
 \end{figure}
Since the boundary components of $V_n$ and $\widetilde{V}_{n-1}$ are all smooth Jordan curves, the map $G_n$ extends to every component of $\partial V_n$ as a differentiable covering. Moreover, as $\widetilde{V}_{n-1}$ is compactly contained in $V_n$, the $G_n$-orbit of any point in $V_n$ falls to or converges to the fixed point $x_n$.
	
\subsubsection{Quasi-conformal extension of $G_{n}$}\label{sec:extend}
\quad\ \ For each $n\geq 1$, we will extend $G_{N+n}: V_{N+n}\to \widetilde{V}_{N+n-1}$ to a $d$-fold quasi-regular map on $\widehat{\mathbb C}$. The following two lemmas will be used in the extension of $G_n$.

\begin{lemma}\label{lem2.3}{\rm(\cite[Lemma 3.1]{CP})}
Let $\gamma_1$ and $\gamma_2$ be two Jordan curves in $\mathbb{C}$, with $q_1\in{\rm int}\widehat{\gamma_1}$ and $q_2\in{\rm int}\widehat{\gamma_2}$. Then given an integer $d_0\geq 1$, there exists a holomorphic proper map $h:{\rm int}\widehat{\gamma_1}\rightarrow{\rm int}\widehat{\gamma_2}$ of degree $d_0$ such that $h(q_1)=q_2$ and $q_1$ is the only possible branch point. Moreover, if $\gamma_1$ and $\gamma_2$  smooth Jordan curves, then $h:\gamma_1\rightarrow\gamma_2$ is smooth.
\end{lemma}
	
\begin{lemma}\label{lem2.4}{\rm(\cite[Lemma 3.2]{CP})}
Let $A_i\subset \mathbb{C}$ be an annulus with the inner boundary $I_i$ and the outer boundary $O_i$, such that $I_i$ and $O_i$ are smooth Jordan curves for $i=1,2$. Suppose that $h_1: I_1\rightarrow I_2$ and $h_2: O_1\rightarrow O_2$ are both $d_0$-fold differentiable covering maps. Then there exists a $d_0$-fold quasi-regular covering map $R: \overline{A_1}\rightarrow\overline{A_2}$  such that $R|_{I_1}=h_1$ and  $R|_{O_1}=h_2$.
\end{lemma}
	
Denote by $\mathcal{E}_f^{\rm crit}$ the set of all critical elements of  ${\cal E }_f$. This set is  finite by Lemma \ref{lem:finite}. Set
\begin{center}
	\begin{math}
		\begin{aligned}
			{\cal E }_f^{*}=\bigcup \{\sigma_f^k(E)\, | &\, k\geq0,\, E\in{\cal E}_f^{\rm crit}  \text{  is preperiodic and its orbit} \\&\text{contains critical periodic components}\}.
		\end{aligned}
	\end{math}
\end{center}
Obviously, ${\cal E }_f^*$ is a finite set and $\sigma_f({\cal E }_f^*)\subset{\cal E }_f^*$. Set $${\cal E}_f^{\rm crit}\cup {\cal E }_f^*=\big\{ E_f^1, E_f^2, \cdots, E_f^{l}\big\}.$$ For each $1\leq k\leq l$, we  choose a preferred point $z^k\in E_f^k\sm \bigcup_{m\geq 0}f^{-m}(0)$.
	
Fix an $n\geq1$. Let $\{B_1,\ldots, B_{i_n}\}$  and  $\{\widetilde{B}_1,\ldots,\widetilde{B}_{j_n}\}$ denote the collection of components of $\mathbb C\sm V_{N+n}$ and  $\mathbb C\sm \widetilde{V}_{N+n-1}$, respectively. Then it is enough to suitably extend the covering map $G_{N+n}:\partial V_{N+n}\to \partial \widetilde{V}_{N+n-1}$ to the interiors of $B_1,\ldots,B_{i_n}$. \vskip 0.1cm
	
By enlarging $N$ if necessary, we may assume that each depth-$0$ puzzle piece contains at most one element of $\{E_f^1, E_f^2, \cdots, E_f^{l}\}$. It follows that each of $\widetilde{B}_1, \ldots, \widetilde{B}_{j_n}$, and hence each of $B_1, \ldots, B_{i_n}$, contains at most one \textbf{marked point} $z_n^k := \pi_n(z^k)$ for $k = 1, \ldots, l$.

For each $j\in\{1,\ldots,j_n\}$, since $\widetilde{V}_{N+n-1}$ is compactly contained in $V_{N+n}$, we can choose a disk $\widetilde{D}_j\Subset {\rm int}\widetilde{B}_j$ with smooth boundary, such that $\widetilde{D}_j$ contains the unique marked point in $\widetilde{B}_j$ (if existing), and that the annulus $\widetilde{A}_j:={\rm int}\widetilde{B}_j\setminus \widetilde{D}_j$ is contained in $V_{N+n}$.
	
For each $i\in\{1,\ldots,i_n\}$, we choose a disk ${D}_i\Subset {\rm int}{B}_i$ with smooth boundary such that ${D}_i$ contains the unique marked point in ${B}_i$ (if existing). Then ${A}_i:={\rm int}{B}_i\setminus {D}_i$ is annulus.
	
For any $i\in\{1,\ldots,i_n\}$, there exists a unique $j=j(i)\in\{1,\ldots,j_n\}$ such that $G_{N+n}:\partial B_i\to \partial\widetilde{B}_j$ is a covering of degree $m_i$. By Lemma \ref{lem2.3}, we obtain a holomorphic proper map $h_{i}: D_i\to \widetilde{D}_j$ of degree $m_i$, such that
\begin{itemize}
\item $h_i$ extends to a smooth covering map from $\partial D_i$ to $\partial\widetilde{D}_j$;
\item if $m_i>1$, then the marked point in $D_i$ is the unique branch point of $h_i$;
\item if both $D_i$ and $\widetilde{D}_j$ contain marked points $z^i_n$ and $z^j_n$ respectively, then $h_i(z^i_n)=z^j_n$.
\end{itemize}
	
By Lemma \ref{lem2.4}, there exists an $m_i$-fold quasi-regular covering $R_i:A_i\to\widetilde{A}_j$ between annuli, such that $R_i$ coincides with $G_{N+n}$ on the outer boundary $\partial B_i$, and coincides with $h_i$ on the inner boundary $\partial D_i$.
	\begin{figure}[htbp]
		\centering
		\includegraphics[width=13cm]{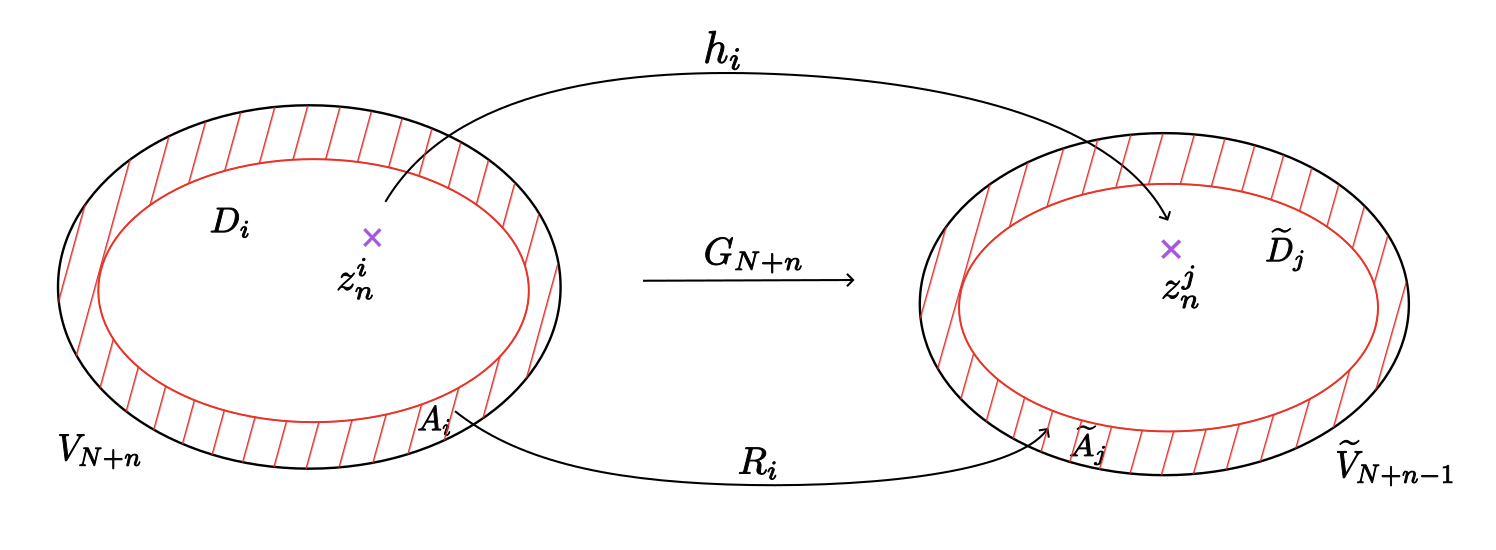}
		\caption{Quasi-conformal surgery. }
		\label{fig:4}
	\end{figure}
	
Thus, we obtain a $d$-fold quasi-regular map $F_n:\widehat{\mathbb C}\to \widehat{\mathbb C}$ defined as
\[F_n(z):=\left\{
\begin{array}{ll}
	G_{N+n}(z), & \hbox{if $z\in\overline{V_{N+n}}$;} \\[2pt]
	R_i(z), & \hbox{if $z\in A_i,\,i=1,\ldots,i_n$;}  \\[2pt]
	h_i(z), & \hbox{if $z\in\overline{D_i},\,i=1,\ldots,i_n$.}
\end{array}
\right.
\]
From the construction, we  see that the $F_n$-orbit of each point in $\widehat{\mathbb C}$ passes through the non-holomorphic part of $F_n$ only once. Then the following (quasiconformal) Surgery Principle, due to Shishikura \cite{Lectures}, can be applied to $F_n$.
	
\begin{lemmaA}\label{lemA}{\rm(\cite[Lemma 15]{Lectures})}
Suppose that $F: \widehat{\mathbb{C}}\rightarrow\widehat{\mathbb{C}}$ is a quasi-regular map and $\sigma$ is a bounded measurable conformal structure such that $F^{*}\sigma=\sigma$ almost everywhere outside a measurable set $X$. If each orbit of $F$ passes through $X$ at most once, then there exists a quasiconformal map $\kappa: \widehat{\mathbb{C}}\rightarrow\widehat{\mathbb{C}}$ such that $R=\kappa\circ h\circ\kappa^{-1}$ is a rational map.
\end{lemmaA}
By this surgery principle, there exists a quasiconformal map $\kappa_{n}:\widehat{\mathbb C}\to \widehat{\mathbb C}$  such that $f_{n}=\kappa_{n}\circ F_{n}\circ \kappa^{-1}_{n}$ is a rational map of degree $d$, and
\[ \text{ the  map $\kappa_{n}$ is conformal in $V_{N+n}$}. \eqno(\ast\ast) \]
Then the following diagrams commute:
\begin{equation}\label{eq:22}
	\begin{tikzcd}
		U_{N+n}\sm\overline{S_{N+n}(t^2)} \arrow[r,"\pi_{N+n}", shorten <=10pt, shorten >=5pt]\arrow[d, "f"]
		&V_{N+n}\arrow[r,"id",hook,shorten <=5pt, shorten >=0pt] \arrow[d, "G_{N+n}"]
		& \widehat{\mathbb{C}}\arrow[r,"\kappa_{n}"]\arrow[d, "F_{n}",shorten <=1pt, shorten >=1pt]
		&  \widehat{\mathbb{C}}\arrow[d, "f_{n}",shorten <=1pt, shorten >=1pt]\\ U_{N+n-1}\sm\overline{S_{N+n-1}(t^2)} \arrow[r,"\pi_{N+n}"]
		& \widetilde{V}_{N+n-1}\arrow[r,"id",hookrightarrow]
		&\widehat{\mathbb{C}}\arrow[r,"\kappa_{n}"]
		&\widehat{\mathbb{C}}
	\end{tikzcd}
\end{equation}
Note that $\kappa_n(V_{N+n})$ is contained in an attracting Fatou domain of $f_n$ and ${\rm deg}(f_n|_{\kappa_n(V_{N+n})})=d$. Thus this attracting Fatou domain is completely invariant.
	
In fact, all objects in \eqref{eq:22} depend on both $n$ and the number $t\in(0,1)$. Here we omit the subscript $t$ because it is fixed. In general, we can write these objects as $\pi_{n,t},\kappa_{n,t},f_{n,t}$ etc. Then the fundamental sequences $\{f_{n,t} \ | \ n\geq1,t\in(0,1)\}$ are constructed.
	
\section{Construction of simple parabolic maps}
This section is devoted to proving Theorem \ref{thm:1.1}. We first verify that the fundamental sequence $\{f_{n,t}\}_{n \geq 1}$ contains a subsequence converging to a simple attracting map $f_t$ for any $t \in (0,1)$ (see Proposition \ref{model}); we then show that a certain subsequence of $\{f_t\}_{t \in (0,1)}$ converges to a simple parabolic map as required in Theorem \ref{thm:1.1} (see Proposition \ref{pro:parabolic}).

For any rational map $g$ with an infinitely-connected completely invariant Fatou domain, we always denote this specific Fatou domain by $U_g$. Recall  the notations  $\mathcal{E}_g=\mathcal{E}_g(U_g)$ and $\sigma_g:\mathcal{E}_g\to \mathcal{E}_g$ from Section \ref{sec:2.1}.
	
\begin{proposition}\label{model}
For any $t\in(0,1)$, there exist a subsequence of $\{f_{n,t}\}_{n\geq1}$  that converges uniformly to a simple attracting map $f_t$ of degree $d$, and a bijection $\xi_t:\mathcal{E}_f\to \mathcal{E}_{f_t}$ satisfying that $\xi_t\circ \sigma_f(E)=\sigma_{f_t}\circ \xi_t(E)$  and $\deg_E\sigma_f=\deg_{\xi_t(E)}\sigma_{f_t}$ for all $E\in {\cal E }_f$.
\end{proposition}
	
The following convergence result given in \cite[Lemma 3.4]{CP} will be used in the proof.
	
\begin{lemma}\label{converge}
Let  $\left\{g_n\right\}_{n\geq 1}$ be a sequence of rational maps with  degree $d\geq 1$, and $D\subset\widehat{\mathbb{C}}$ a non-empty open set. If $g_n$ converges uniformly  to a map $g$ on $D$, then $g$ is a restriction of a rational map of degree $d_0\leq d$. Moreover, $d_0=d$ implies that $g_n$ converges uniformly to $g$ on $\widehat{\mathbb{C}}$.
\end{lemma}
	
\begin{proof}[Proof of Proposition \ref{model}]
The outline is similar to that of \cite[Proposition 1.1]{CP}, but with more complexity  involved since $U$ is a parabolic Fatou domain and we have performed a plumbing surgery to $(f,U)$.
		
Fix any $t\in(0,1)$. For simplicity of  the discussion below, any object $Y_n$ refers to a double-subscript object $Y_{n,t}$, and the convergence $Y_n\to Y_t$ means $Y_{n,t}\to Y_t$ as $n\to \infty$.
		
For each $n\geq 1$, define $$\psi_{n}=\kappa_{n}\circ\pi_{N+n}\ \text{ on }\ \widehat{\mathbb C}\setminus \overline{S_{N+n}(t^2)},$$ where $N$ is the number  such that $\mathbb C\setminus U_N$ is the union of all depth-$0$ parabolic puzzle pieces of $(f,U)$. Then $\psi_{n}$ is holomorphic in $U_{N+n}\sm\overline{S_{N+n}(t^2)}$.
By the statements ($\ast$) and ($\ast\ast$), we further obtain that
		
\begin{claim}\label{claim:univalent}
The map $\psi_n$ is univalent on both $U_{N+n}\sm\overline{S_{N+n}(t)}$ and $D\cap U_{N+n}$, where $D$ is any component of $S_{N+n}(1)\setminus\overline{S_{N+n}(t^2)}$.
\end{claim}
		
Normalize $\psi_{n}$  such that it fixes $\infty$ and two other points near $\infty$. Then $\left\{ \psi_{n}\right\}_{n\geq 1}$  is a normal family on $U_{N+k}\sm\overline{S_{N+k}(t^2)}$ for each $k\geq 1$. By Cantor's diagonal method, there exists a subsequence $\left\{ \psi_{n_k}\right\}_{k\geq1}$  that converges locally and uniformly to a holomorphic map $\psi_t$ on
\[\bigcup_{k\geq 1}\big(U_{N+k}\sm\overline{S_{N+k}(t^2)}\big)=U\sm\bigcup_{k\geq 1} \overline{S_{N+k}(t^2)}=U\sm\bigcup_{k\geq 0} \overline{S_k(t^2)}.\]

By Claim \ref{claim:univalent}, $\psi_t$ is univalent on both $U \setminus \bigcup_{k \geq 0} \overline{S_k(t)}$ and each component of $S_k(1) \setminus \overline{S_k(t^2)}$ for every $k \geq 0$. Fix any $z \in S_0 \setminus \overline{S_0(t^2)}$. For any $n \geq 0$, it follows from the definitions of $\pi_n$ and $\tau_n$ that $\pi_n(\tau_0(z)) = \pi_n(\tau_n(z)) = \pi_n(z)$. Thus, for each sufficiently large $k$, we have $\psi_{n_k} \circ \tau_0(z) = \psi_{n_k}(z)$. Consequently,
\begin{equation}\label{eq:33}
\psi_t\circ\tau_0(z)=\psi_t(z)  \quad \text{for all }  z\in S_0\setminus\overline{S_0(t^2)}.
\end{equation}
		
We conclude from  Lemma \ref{converge} that $f_{n_k}$, which coincides with $\psi_{n_k}\circ f\circ\psi_{n_k}^{-1}$ in a neighborhood of $\infty$,  converges uniformly in $\widehat{\mathbb{C}}$ to a rational map $f_t$ of degree $d$, and
\begin{equation}\label{eq:44}
    \psi_t\circ f(z)=f_t\circ \psi_t(z) \quad \text{for all }  z\in U\sm\bigcup_{k\geq 0} \overline{S_{k}(t^2)}.
\end{equation}
Define the domain $V_t=\psi_t\big( U\sm\bigcup_{k\geq 0} \overline{S_{k}(t^2)}\big)$. Then formula \eqref{eq:44} implies that $f_t(V_t)=V_t$ and ${\rm deg}(f_t|_{V_t})=d$. So $V_t$ is contained in a completely invariant Fatou domain $U_{f_t}$ of $f_t$. We will show that $U_{f_t}$ is an attracting Fatou domain of $f_t$. \vskip 0.1cm
		
Recall that  $L_\pm(s)=\partial S_\pm(s)\sm\{0\}$ for $s\in(0,1)$. Since $f:L_\pm(s)\to L_\pm(s)$ is a homeomorphism by \eqref{eq:11}, we can parameterize $$L_+(s)=L^s_+:(-\infty,+\infty)\to S_0\setminus \overline{S_0(t^2)}$$ such that $f\circ L_+^s(x)=L_+^s(x+1)$ for $x\in(-\infty,+\infty)$ and $\lim_{x\to +\infty}L_\pm^s(x)=\lim_{x\to -\infty}L_\pm^s(x)=0$.
		
\begin{claim}\label{claim:2}
There exist fixed points $a_t,b_t\in\mathbb C$ such that, for any $s\in(t^2,1)$, the open arc $\g_s=\psi_t\circ L_+^s:(-\infty,+\infty)\to V_t$ satisfies  $\lim_{x\to +\infty}\g_s(x)=a_t$ and $\lim_{x\to -\infty}\g_s(x)=b_t$. As a consequence, $B_0:=\psi_t(S_0\setminus \overline{S_0(t^2)})=\bigcup_{s\in(t^2,1)}\g_s$ is a disk in $U_{f_t}$.
\end{claim}
		
\begin{proof}
Since $\psi_t$ is univalent on $S_0 \setminus \overline{S_0(t^2)}$, the domain $B_0 \subset U_{f_t}$ is simply connected. Fix any $s \in (t^2, 1)$. By the definition of $\gamma_s$ and \eqref{eq:44}, we have $f_t \circ \gamma_s(x) = \gamma_s(x+1)$ for $x \in (-\infty, +\infty)$. Write $\gamma_s = \bigcup_{k=-\infty}^{+\infty} I_k$, where $I_k := \gamma_s[k, k+1]$. Then $f_t(I_k) = I_{k+1}$ by the parameterization of $\gamma_s$.

Since $f_t : B_0 \to B_0$ is conformal by \eqref{eq:44}, the hyperbolic lengths of $I_k$ for $k \in \mathbb{Z}$ are the same. Note also that $\gamma_s$ tends to $\partial B_0$ as $x \to \pm\infty$. It then follows that $\mathrm{diam}(I_k) \to 0$ as $k \to \infty$.

Set $K := \bigcap_{k \geq 0} \overline{\gamma_s(k, \infty)}$. Then $K$ is connected and contained in $\partial B_0$. Let $w \in K$. There exist a subsequence $\{k_m\}$ and points $w_{k_m} \in I_{k_m}$ such that $w_{k_m} \to w$ as $k_m \to \infty$. Fix any neighborhood $D$ of $w$ in $\mathbb{C}$. Since $\mathrm{diam}(I_k) \to 0$ as $k \to \infty$, the arc $I_{k_m}$ is contained in $D$ for all sufficiently large $k_m$. Note that $f_t$ maps one endpoint of $I_{k_m}$ to the other. Then $f_t(D) \cap D \neq \emptyset$. The arbitrariness of $D$ implies that $f_t(w) = w$. Since $K$ is connected but the fixed points of $f_t$ are discrete, it follows that $K$ is a singleton, denoted by $a_t$, and obviously $f_t(a_t) = a_t$.
			
Let $s' \in (t^2, 1)$ be any other number. It is known that $\lim_{x \to +\infty} \gamma_{s'}(x) = a_t'$. We will show $a_t = a_t'$. To see this, let $\delta_0 \subset B_0$ be an arc joining $\gamma_s(0)$ and $\gamma_{s'}(0)$. As before, it holds that $\mathrm{diam}(\delta_k) \to 0$ as $k \to \infty$, where $\delta_k := f_t^k(\delta_0)$ is an arc joining $\gamma_s(k)$ and $\gamma_{s'}(k)$. This implies $a_t = a_t'$, since $\gamma_s(k) \to a_t$ and $\gamma_{s'}(k) \to a_t'$.

By a similar argument as above, we can prove that $\gamma_s(x)$ converges to a fixed point $b_t$ as $x \to -\infty$ for any $s \in (t^2, 1)$. Then the claim is proved.
\end{proof}
		
Since $\tau_0(L_+(s))=L_-(t^2/s)$ (item (1) in Section \ref{sec:plumbing}), relation \eqref{eq:33} implies that $\psi_t(S_0\setminus \overline{S_0(t)})=B_0\setminus \g_t$. Note that $\beta := \partial U_0 \setminus S_0(t) \subset U$ is an arc satisfying $\beta(0) \in L_+(t)$, $\beta(1) \in L_-(t)$, and $\tau_0(\beta(0)) = \beta(1)$. Then $\psi_t(\beta(0)) = \psi_t(\beta(1)) \in \gamma_t$ by \eqref{eq:33}. Note also that $\psi_t$ is univalent on $U \setminus \overline{S_0(t)}$. Then $\alpha := \psi_t(\beta) \subset U_{f_t}$ is a Jordan curve that bounds a disk $W_0 \ni a_t$. Combining Claim \ref{claim:2}, we have
\[\psi_t(U_0 \setminus \overline{S_0(t^2)}) = W_0 \setminus \{a_t\} \quad \text{and} \quad \psi_t(U_0 \setminus \overline{S_0(t)}) = W_0 \setminus \overline{\gamma_t},\] see Figure \ref{fig:5}.
\begin{figure}[htbp]
	\centering
	\includegraphics[width=15cm]{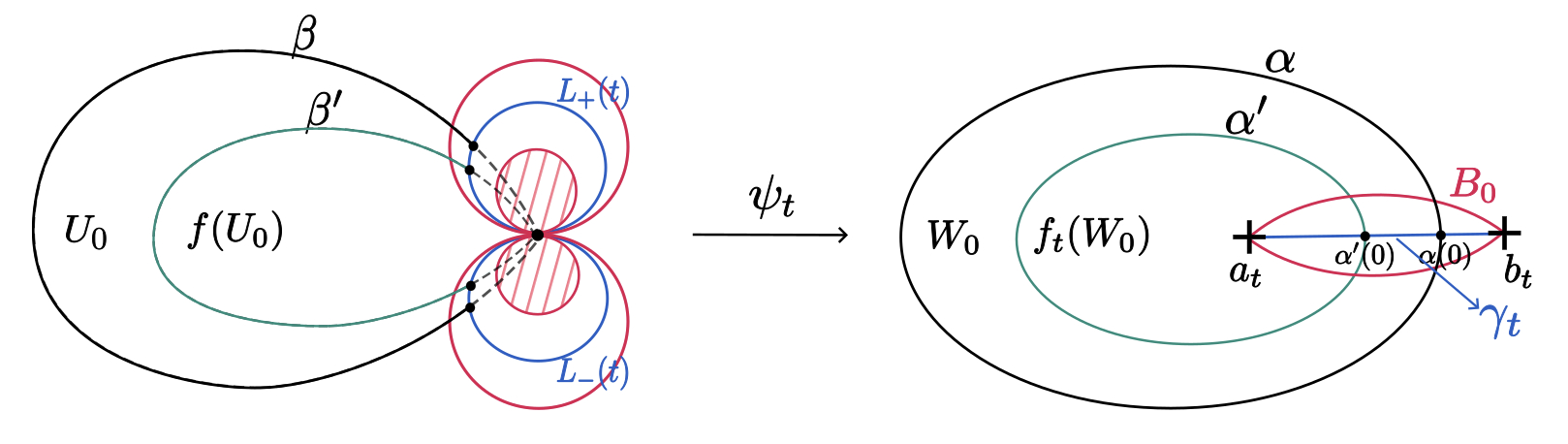}
	\caption{Image of $\psi_t$ on $U_0\setminus \overline{S_0(t^2)}$}
	\label{fig:5}
\end{figure}

Similarly, the arc $\beta':=\partial f(U_0)\setminus S_0(t)$ is mapped by $\psi_t$ to a Jordan curve $\alpha'\subset U_{f_t}$ disjoint from $\alpha$. Note that $\beta'(0)=f(\beta(0))$. It then follows from \eqref{eq:44} that $f_t(\alpha(0))=\alpha'(0)\in \g_t$. So $\alpha'(0)$, and thus the entire $\alpha'$, is contained in $W_0$, as $f_t\circ \g_t(x)=\g_t(x+1)$ for any $x\in(-\infty,+\infty)$. Using \eqref{eq:44} again, we conclude that $f_t(W_0)\Subset W_0$. It follows that $a_t$ is an attracting fixed point.

Therefore, $U_{f_t}$ is a completely invariant attracting Fatou domain of $f_t$.		

For any $k\geq0$, denote by $W_k$ the component of $f^{-k}_t(W_0)$ containing $a_t$. By the argument above, it follows that for any $k\geq0$,
\[\label{***} \psi_t\big(U_k\sm\overline{S_k(t^2)}\big)=W_k\setminus f_t^{-k}(a_t)\text{ and }\psi_t : U_k\sm\overline{S_k(t)}\to W_k\setminus f_t^{-k}(\overline{\g_t})\text{ is conformal}.\eqno(\ast\! \ast \! \ast) \]

To establish a bijection between $\mathcal{E}_f$ and $\mathcal{E}_{f_t}$, we define the depth-$n$ (attracting) puzzle $\mathcal{Q}_n$ of $(f_t, U_{f_t})$ to be the collection of the components of $\mathbb{C} \setminus W_{N+n}$ for each $n \geq 0$. In this case, each depth-$(n+1)$ puzzle piece is a closed disk and is contained in the interior of a depth-$n$ puzzle piece. Moreover, since $U_{f_t}$ is completely invariant, the map $f_t$ maps each depth-$(n+1)$ puzzle piece onto a depth-$n$ one as a branched covering.
		
By the mapping property of $\psi_t$ given in $(\ast\! \ast \! \ast)$, we can define a bijection $\xi_n:\mathcal P_n\to \mathcal Q_n$ between puzzles of depth $n$ for each $n\geq0$, such that
\[Q_n:=\xi_n(P_n)\ \text{ if }\ \psi_t(\partial P_n\setminus \overline{S_{N+n}(t^2)})=\partial Q_n.\]
This naturally induces a map $\xi_t: \mathcal E_f\to \mathcal E_{f_t}$ defined as
\[\xi_t(E)=\partial \big(\bigcap_{n\geq0}\xi_n\big(P_n(E)\big)\big)  \quad \text{for all }  E\in \mathcal E_f.\]
		
\begin{claim}\label{claim:correspond}
The map $\xi_t$ satisfies the properties  in Proposition \ref{model}.
\end{claim}

\begin{proof}
Since each $\xi_n$ is bijective, it follows directly that $\xi_t: \mathcal E_f\to \mathcal E_{f_t}$  is a bijection.
			
For each sufficiently large $n$, formula \eqref{eq:44} implies
$$f_t\Big( \psi_t\big(\partial P_{n+1}(E) \sm\overline{S_{N+n+1}(t)}\big)\Big)=\psi_t\Big(\partial P_{n}(\sigma_f(E)) \sm\overline{S_{N+n}(t)}\Big).$$
Hence $f_t\Big(\xi_{n+1} \big(P_{n+1}(E)\big)\Big)=\xi_n \Big(P_{n}\big(\sigma_f(E)\big)\Big)$ by definition of $\xi_n$. It implies that  $$f_t\big(\xi_t({E})\big)=f_t\Big(\partial\bigcap_{n\geq 0}\xi_{n+1}\big(P_{n+1}(E)\big)\Big)=\partial \bigcap_{n\geq 0}\xi_n\Big (P_{n}\big(\sigma_f(E)\big)\Big)=\xi_t\big(\sigma_f(E)\big).$$ Consequently, $\sigma_{f_t}\circ \xi_t=\xi_t\circ \sigma_f$ for each  $E\in{\cal E }_f$.\vskip 0.1cm
			
By the univalent property  of $\psi_t$, we obtain $\deg_E\sigma_f=\deg_{\xi_t(E)}\sigma_{f_t}$ for all  $E\in{\cal E }_f$.
\end{proof}

It remains to verify that $f_t$ is a simple attracting map.\vskip 0.1cm
		
With similar notations to those in Section \ref{sec:extend}, we denote by $\mathcal{E}_{f_t}^{\rm crit}$  the set of all critical components in  $ {\cal E }_{f_t}$, and set
\begin{center}
	\begin{math}
		\begin{aligned}
			{\cal E }_{f_t}^{*}=\bigcup \{\sigma_{f_t}^k(E)\, | &\, k\geq0,\, E\in{\cal E}_{f_t}^{\rm crit}  \text{  is preperiodic and its orbit } \\&\text{ contains critical periodic components}\}.
		\end{aligned}
	\end{math}
\end{center}		
From Claim \ref{claim:correspond}, we conclude that the cardinality of ${\cal E}_{f_t}^{\rm crit}$ (resp. $\mathcal{E}_{f_t}^*$) coincides with that of ${\cal E}_f^{\rm crit}$ (resp. $\mathcal{E}_f^*$); we denote these cardinalities by $l_1$ (resp. $l_2$).

\begin{claim}\label{claim:3}
	The filling of each element of $\mathcal E_{f_t}^*$ contains a unique critical or postcritical point.
\end{claim}
\begin{proof}
Recall  that ${\cal E}_f^{\rm crit}\bigcup {\cal E }_f^*=\big\{ E_f^1, E_f^2, \cdots, E_f^{l}\big\}$. We have assigned a preferred point $z^j\in \widehat{E}_f^j$ for each $j=1,\ldots,l$. Define an injection $\alpha_{n_k}: {\cal E}_f^{\rm crit}\bigcup {\cal E }_f^*\to \mathbb C$ for each $k\geq1$ by $\alpha_{n_k}(E_f^j):=\psi_{n_k}(z^j),j=1,\ldots,l$. By the construction of $f_n$  (see the graph \eqref{eq:22}), it holds that
\begin{equation}\label{eq:66}
	{\rm deg}(f_{n_k}|_{\alpha_{n_k}(E_f^j)})={\rm deg}_{E_f^j}\sigma_f,\forall\, j\in\{1,\ldots,l\}\ \text{ and }\ \alpha_{n_k}\circ \sigma_f=f_{n_k}\circ \alpha_{n_k}\text{ on }\mathcal{E}_f^*.
\end{equation}
By taking a subsequence, we may assume that $\alpha_{n_k} \to \alpha_t: \mathcal{E}_f^{\rm crit} \cup \mathcal{E}_f^* \to \mathbb{C}$ as $k \to \infty$. Then $\alpha_t$ is an injective map satisfying
\begin{equation}\label{eq:77}
	\alpha_t \circ \sigma_f(E) = f_t \circ \alpha_t(E) \quad \text{for all } E \in \mathcal{E}_f^*.
\end{equation}
Let $\alpha_t(\mathcal{E}_f^{\rm crit}) = \{w^1, w^2, \dots, w^{l_1}\}$. Then $w^i \neq w^j$ for $1 \leq i \neq j \leq l_1$. It is easy to check that each $w^i$ is a critical point of $f_t$ with multiplicity at least $(\deg_{E_f^i} \sigma_f - 1)$ for $1 \leq i \leq l_1$.

Let  $M$ denote the number of critical points of $f$ in $U$ (counted with multiplicity). According to the construction of $f_n$ and the degree property in \eqref{eq:66} , we have
\[\sum_{E\in\mathcal E_f^{\rm crit}} \Big({\rm deg}\big(f_{n_k}|_{\alpha_{n_k}(E)}\big)-1\Big)=2d-2-M.\]
The univalence of $\psi_t$ implies that ${f_t}$ has $M$ critical points in $U_{f_t}$. Thus the critical points of $f_t$ outside $U_{f_t}$ are precisely $w^{1}, w^{2},\dots, w^{l_1}$. Applying Lemma \ref{lem:finite} to $(f_t,U_{f_t})$, we can see that each $\widehat{E_t}$ with $E_t\in\mathcal E_{f_t}^{\rm crit}$ contains exactly one  critical point of $f_t$ outside $U_{f_t}$.

Set $Z_t^* := \alpha_t(\mathcal{E}_f^*)$. By the equation \eqref{eq:77}, we have $f_t(Z_t^*) \subset Z_t^*$. Since $U_{f_t}$ is completely invariant, for any $E_t \in \mathcal{E}_{f_t}^*$, it holds that $f_t(\widehat{E_t}) = \widehat{\sigma_t(E_t)}$. Combining these two facts and recalling that every $E_t \in \mathcal{E}_{f_t}^*$ is an element of the orbit of some critical component (by definition of $\mathcal{E}_{f_t}^*$), we conclude that $\widehat{E_t}$ contains at least one point in $Z_t^*$. Since $\alpha_t$ is injective, $\# Z_t^*=\# \mathcal{E}_{f_t}^*$. Thus $\widehat{E_t}$ contains exactly one point of $Z_t^*$, and therefore one critical or postcritical point of $f_t$.
\end{proof}

Since $U_{f_t}$ is completely invariant, by Lemma \ref{lem:component} (1), an element $E \in \mathcal{E}_{f_t}$ is a singleton if its orbit under $\sigma_{f_t}$ does not enter $\mathcal{E}_{f_t}^*$. Because the orbit of each point in $Z_t^*$ contains a periodic critical point, we apply Claim \ref{claim:3} and Lemma \ref{lem:component} (2) to conclude that each $E_t \in \mathcal{E}_{f_t}^*$ is a quasi-circle. Thus $f_t$ is a simple attracting map.
\end{proof}
	
\begin{proposition}\label{pro:parabolic}
There exists a decreasing sequence $\{t_j\}_{j\geq1}\subset (0,1)$  converging to $0$, such that the simple attracting maps $\{f_{t_j}\}_{j\geq1}$ obtained in Proposition \ref{model} converge to a simple parabolic map $g$, and  $(f,U)$ is conformally conjugate to $(g,U_g)$.
\end{proposition}
\begin{proof}
In the proof of  Proposition \ref{model},  we constructed a univalent map $\psi_{t}: U\sm\bigcup_{k\geq 0} \overline{S_k(t)}\to \mathbb C$ and a simple attracting map $f_{t}$ of degree $d$ for any $t\in(0,1)$, such that
\begin{equation}\label{eq:88}
f_{t}\circ\psi_{t}(z)=\psi_{t}\circ f(z)  \quad \text{for all }  z\in U\setminus\bigcup_{k\geq0}\overline{S_{k}(t)},
\end{equation}
and $\psi_t$ fixes $\infty$ and two other points near $\infty$. Moreover, by  Claim \ref{claim:correspond}, there exists a bijection $\xi_t:\mathcal E_f\to \mathcal E_{f_t}$ such that $\xi_t\circ \sigma_f=\sigma_{f_t}\circ\xi_t$ on $\mathcal E_f$ and ${\rm deg}_E\sigma_f={\rm deg}_{\xi_t(E)}\sigma_{f_t}$ for any $E\in\mathcal E_f$. 

Notice that $\{\psi_{s}\}_{s\in(0,1)}$ forms a normal family on $U\sm\bigcup_{k\geq 0} \overline{S_k(s)}$ for any $s\in(0,1)$. By taking a subsequence, it follows that $\{\psi_{t_j}\}_{j\geq1}$ converges locally and uniformly  to a univalent map $\psi:U\to\widehat{\mathbb C}$. According to Lemma \ref{converge} and the equation \eqref{eq:88}, $f_{t_j}$ converges uniformly to a rational map $g$ of degree $d$ as $j\to\infty$, such that $\psi\circ f=g\circ \psi$ in $U$. As a result, $(f,U)$ and $(g,U_g)$ are conformally conjugate with  $U_g:=\psi(U)$. \vskip 0.1cm

As $g$ is a rational map of degree $d$, $g^{-1}(U_g) = U_g$. $U_g$ is contained in a Fatou domain $\widetilde{U}_g$ of $g$.  If $\widetilde{U}_g$ is a Siegel disk or Herman ring, then $g|_{\widetilde{U}_g}$ is holomorphically conjugate to an irrational rotation of degree 1, while $\deg(g|_{\widetilde{U}_g}) = d \geq 2$ since $U_g \subset \widetilde{U}_g$ is completely invariant. If $\widetilde{U}_g$ is an attracting domain, then $\widetilde{U}_g/\left\langle g\right\rangle $ is a torus.  But $U_g/\left\langle g\right\rangle $ is conformally isomorphic to $U/\left\langle f\right\rangle$ which is an infinite cylinder. So $\widetilde{U}_g$ cannot be attracting. If $U_g\subsetneq \widetilde{U}_g$, then $U_g/\left\langle g\right\rangle $ is conformally isomorphic to a subset of $\mathbb{C}^{*}$, leading to a contradiction.
Thus $U_g$ is a completely invariant parabolic Fatou domain of $g$. 

It remains to check that $g$ is a simple parabolic map.	\vskip 0.1cm
		
Recall that $ {\cal E }_g$ denotes the collection of all components of $\partial U_g$. Since $\psi$ is a conformal conjugation from $f:U\to U$ onto $g:U_g\to U_g$, we set $\Omega_k:=\psi(U_k)$ for each $k\geq0$, and define the depth-$n$ puzzle $\mathcal P_n(g)$ of $g$ as the collection of components of $\mathbb C\setminus \Omega_{N+n}$ for each $n\geq0$. Clearly, $\psi$ induces a one-to-one correspondence between $\mathcal P_n(f)$ and $\mathcal P_n(g)$. With a similar argument to Claim \ref{claim:correspond}, we obtain a bijection $\xi:\mathcal E_f\to \mathcal E_g$ such that $\xi\circ \sigma_f=\sigma_g\circ \xi$ on $\mathcal E_f$ and ${\rm deg}_E\sigma_f={\rm deg}_{\xi(E)}\sigma_{g}$ for all $E\in\mathcal E_f$. \vskip 0.1cm
		
As before, define by $\mathcal E_g^{\rm crit}$ the collection of critical elements of $\mathcal E_g$ and
\begin{center}
	\begin{math}
		\begin{aligned}
			{\cal E }_{g}^{*}=\bigcup \{\sigma_{g}^k(E)\, | &\, k\geq0,\, E\in{\cal E}_{g}^{\rm crit}  \text{  is preperiodic and its orbit } \\&\text{ contains critical periodic components}\}.
		\end{aligned}
	\end{math}
\end{center}		
		
With a similar argument to that in the proof of Claim \ref{claim:3} (by replacing $f_{n_k}, \psi_{n_k},f_t$ and $\xi_t$ there, with $f_{t_j}, \psi_{t_j}, g$ and $\xi$, respectively), we can show that for any $E\in\mathcal E_g^*$, it does not contain critical points of $g$ and int$\widehat{E}$ contains exactly one critical or postcritical point of $g$. Thus $g$ is a simple parabolic map by Lemma \ref{lem:component}.
\end{proof}
	
\begin{proof}[Proof of Theorem \ref{thm:1.1}]
It follows directly from Propositions \ref{model} and \ref{pro:parabolic}.
\end{proof}
	
\section{Perturbation of parabolic maps with Cantor Julia sets}
	
In this section, we will prove Theorems \ref{thm:middle} and \ref{thm:1.2}.
	
Let $f$ be a simple parabolic map of degree $d$. Fix any $t\in(0,1)$. Let $f_t$ be the simple attracting map which is obtained by Proposition \ref{model} and based on $(f,U_f)$. To prove Theorem \ref{thm:middle}, we define new puzzles for $f$ and $f_t$ slightly different from those  given in Section \ref{sec:puzzle} and Section 4, respectively.
	
As in Section \ref{sec:puzzle}, let $U_0\subset U_f$ be an (regular) attracting petal of the parabolic fixed point for $f$, and $S_0$ be the union of sepals attached to the parabolic point. Then $$U^*_0:=U_0\cup S_0$$ is still a disk. For every $n\geq0$, denote by $U^*_{n}$ the component of $f^{-n}(U^*_0)$ containing $U^*_0$. There exists an integer $N\geq 1$ such that  $f^{-i}(U^*_{N})$ has only one component for any $i\geq 0$.
	
For each $n \geq 0$, we define the new depth-$n$ puzzle $\mathcal{P}_n^*$ for $f$ as the collection of all components of $\mathbb{C} \setminus \overline{f^{-n}(U_N^*)}$. The new puzzle $\mathcal{P}_n^*$ differs only slightly from $\mathcal{P}_n$: it satisfies properties (P1)-(c) and (d) of Lemma \ref{puzzle}, though its remaining characteristics diverge from $\mathcal{P}_n$. Specifically, all puzzle pieces in $\mathcal{P}_n^*$ are disks, and for any two puzzle pieces $P_n^*$ and $P_{n+1}^*$,  if $\partial P_{n}^*\cap\partial P_{n+1}^*\neq\emptyset$, then $\partial P_n^* \cap \partial P_{n+1}^* \subset f^{-(N+n)}(\partial S_0) \cap \overline{P_{n+1}^*}$. Moreover, for any $E \in \mathcal{E}_f$ and $n \geq 0$, there exists a unique $P_n^*(E) \in \mathcal{P}_n^*$ such that $E \subset \overline{P_n^*(E)}$, and it holds that $\bigcap_{n \geq 0} \overline{P_n^*(E)} = \widehat{E}$.\vskip 0.1cm
	
In the proof of Proposition \ref{model}, we obtain a holomorphic map $\psi_t$ on $U_f \setminus \bigcup_{k \geq 0} \overline{S_k(t^2)}$ that is univalent on both $U \setminus \overline{S_k(t)}$ and each component of $S_k \setminus \overline{S_k(t^2)}$ for all $k \geq 0$, satisfying
\[\psi_t \circ f(z) = f_t \circ \psi_t(z) \quad \text{for all } z \in U_f \setminus \bigcup_{k \geq 0} \overline{S_k(t^2)}.\]
By Claim \ref{claim:2}, $B_0=\psi_t(S_0\setminus \overline{S_0(t^2)})\subset U_{f_t}$ is a disk, such that $\partial B_0\cap \partial U_{f_t}$ is a repelling fixed point $b_t$,  and $f_t:B_0\to B_0$ is conformal.
	
Recall that $W_0=\psi_t(U_0\sm\overline{S_0(t^2)})\sm\{a_t\}$, where $a_t$ is the attracting fixed point in $U_{f_t}$ . Then $$W_0^*:=W_0\bigcup B_0\bigcup \{a_t\}$$ is a disk and $f_t(W_0^*)\subset W_0^*$. For every $n\geq0$, denote by $W^*_{n}$ the component of $f^{-n}_t(W^*_0)$ containing $a_t$. For each $n\geq 0$, we define the new depth-$n$ puzzle $\mathcal{Q}^*_n$ for $f_t$ to be the collection of all components of $\omC\sm \overline{f_t^{-n}({W^*_{N}})}$. Then all puzzle pieces are disks. 
	
By the construction above, we can see that $\psi_{t}:\partial U_n^*\to\partial W_n^*$ is a homeomorphism for every $n\geq0$, see Figure \ref{fig:6}. Then $\psi_t$ induces a bijection $\eta_n:\mathcal P_n^*\to \mathcal Q_n^*$ such that
	\[\eta_n(P_n^*)=Q_n^*\ \text{ if }\ \psi_t(\partial P_n^*)=\partial Q_n^*\]
As a consequence, we have a map $\eta: \mathcal{E}_f \to \mathcal{E}_{f_t}$ defined by
	\[{\eta(E)} = \partial \bigcap_{n \geq 0}\overline{\eta_n(P^*_n(E))}  \quad \text{for all } E\in \mathcal E_f.\]
Using a method similar to that in the proof of Claim \ref{claim:correspond}, we can show that $\eta: \mathcal{E}_f \to \mathcal{E}_{f_t}$ is a bijection satisfying
\begin{equation}\label{eq:99}
	\eta \circ \sigma_f(E) = \sigma_{f_t} \circ \eta(E) \quad \text{and} \quad \deg_E \sigma_f = \deg_{\eta(E)} \sigma_{f_t}
	\quad \text{for all } E \in \mathcal{E}_f.
\end{equation}
\begin{figure}[htbp]
	\centering
	\includegraphics[width=13cm]{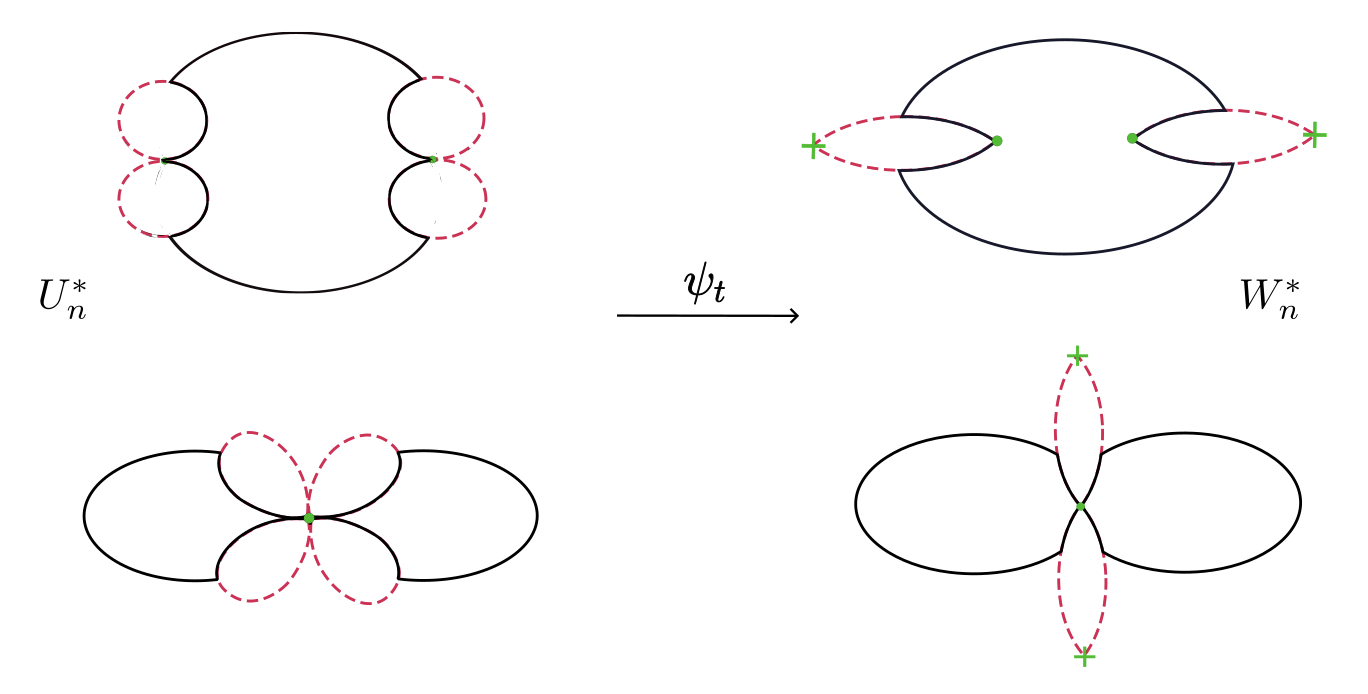}
	\caption{The boundary correspondence between $U_n^*$ and $W_n^*$}
	\label{fig:6}
\end{figure}
		
\begin{proof}[Completion of the proof of Theorem \ref{thm:middle}]
For any $n\geq1$ and $E\in\mathcal E_f$, we will define  puzzles $\mathcal X_n(E)$ and $\mathcal Y_n(E)$ of depth $n$ for $(f,E)$ and $(f_t,\eta(E))$, respectively.
		
If $E\in\mathcal E_f$ is a singleton, define $\mathcal X_n(E)=\big\{P_n^*(E)\big\}$ and $\mathcal Y_n(E)=\big\{Q_n^*(E)\big\}$ for each $n\geq1$.
		
Now suppose that $E\in\mathcal E_f$ is not a singleton. Then $E$ is $\sigma_f$-preperiodic  by Lemma \ref{lem:component} (1). We first consider the periodic case. Without loss of generality, assume that $f(E)=E$. Let $x^*$ be the unique parabolic fixed point of $f$ and $d_0:=\deg(f|_{\widehat{E}})$.
		
Denote by $\Omega_E$ the superattracting Fatou domain bounded by $E$, and by $\Omega_E(s)$ the set of points in $\Omega_E$ with potential (under Bottcher coordinate) less than $s$ ($s>0$). Similarly, we can define $\Omega_{\eta(E)}$ and $\Omega_{\eta(E)}(s)$. Using the Bottcher coordinates, we obtain a conformal map $\chi_E:\Omega_E\to \Omega_{\eta(E)}$ with $\chi_{f(E)}\circ f=f_t\circ\chi_E $. As $E$ and $\eta(E)$ are both Jordan curves, the map $\chi_E$ extends to a homeomorphism $\chi_E:\overline{\Omega_E}\to \overline{\Omega_{\eta(E)}}$. There are finitely many choices for these $\chi_E$. We will specify one by the following claim.
		
\begin{claim}\label{claim:6}
Suppose that $E$ avoids the unique parabolic point of $f$, and let $y\in E$ be a fixed point of $f$. Then there exists an open arc $\alpha_0\subset U_f\setminus \bigcup_{k\geq0}\overline{S_k}$ such that $\alpha_0$ joins a boundary point of $P_0^*(E)$ to $y$ and $\alpha_0\subset f(\alpha_0)$.
\end{claim}
\begin{proof}
We know that $E$ is fixed and $y$ is a repelling fixed point. Let $U_y$ denote the linearization domain of $y$. There exists an integer $n_0 \geq 0$ sufficiently large such that $P^*_{n_0}(E) \cap U_y \neq \emptyset$ and $P^*_{n_0}(E) \sm \overline{P^*_{n_0+k}(E)}$ is an annulus for an integer $k\geq 1$. 

Choose a point $w_0 \in U_y \cap \partial P_0^*(E)$ such that $w_0 \notin f^{-N}(x^*)$. There exists a point $w_k \in U_y \cap \partial P_{n_0+k}^*(E)$ with $w_k \in f^{-k}(w_0)$. Consider an open arc $\delta_0$ connecting $w_0$ and $w_k$ such that $\delta_0 \cap P(f) = \emptyset$. Since each non-singleton component of $\mathbb{C} \setminus U_f$ contains at most one postcritical point of $f$, and $\big( \overline{P_0^*(E)} \cap U_f \big) \cap P(f) = \emptyset$, we can select $\widetilde{\delta}_0$ in the homotopy class of $\delta_0$ rel $P(f)$ (connecting $w_0$ and $w_k$) such that $\widetilde{\delta}_0 \subset U_f$. We may thus assume $\delta_0 \subset U_f \cap U_y$. Since $\bigcup_{k \geq 0} \overline{S_k} \cap P(f) = \emptyset$, we can further choose $\delta_0$ to satisfy $\delta_0 \cap \bigcup_{k \geq 0} \overline{S_k} = \emptyset$. It follows that $\delta_0 \subset \big( U_f \setminus \bigcup_{k \geq 0} \overline{S_k} \big) \cap U_y.$

The set $\bigcup_{n \geq 0} f^{-n}(\overline{\delta_0})$ is a fixed ray in $\left( U_f \setminus \bigcup_{k \geq 0} \overline{S_k} \right) \cap U_y$. Let
\[\alpha_0 := \big( \bigcup_{n \geq 0} f^{-n}(\overline{\delta_0}) \big) \setminus \{w_0\}.\]
Then $\alpha_0$ is an open arc satisfying $\alpha_0 \subset U_f \setminus \bigcup_{k \geq 0} \overline{S_k}$ and $\alpha_0 \subset f(\alpha_0)$.\vskip 0.1cm

Take a disk $D \subset P_0^*(E) \cap U_f$ that contains $\delta_0$. Since all non-singleton Julia components of $f$ contain no critical points, we have $\big( \overline{P_0^*(E) \cap U_f} \big) \cap P(f) = \emptyset$, which implies $\overline{D} \cap P(f) = \emptyset$. Let $C_n$ denote the maximum diameter of the components of $f^{-n}(D)$. By the Shrinking Lemma in \cite{L}, $C_n \to 0$ as $n \to \infty$. Hence, $\alpha_0$ connects $w_0$ to $y$.
\end{proof}

Now we specify the map $\chi_E:\Omega_E\to \Omega_{\eta(E)}$. If $E$ contains the parabolic fixed point $x^*$, then $x^* \in \partial P_0^*(E)$. Since $\psi_t(\partial P_0^*(E)) = \partial Q_0^*(\eta(E))$, it follows that $\psi_t(x^*) \in \eta(E) \cap \partial Q_0^*(\eta(E))$, and $\psi_t(x^*)$ is a repelling fixed point of $f_t$. We therefore choose $\chi_E$ such that $\chi_E(x^*) = \psi_t(x^*)$ (see Figure \ref{fig:7}).
\begin{figure}[htbp]
\centering
\includegraphics[width=16cm]{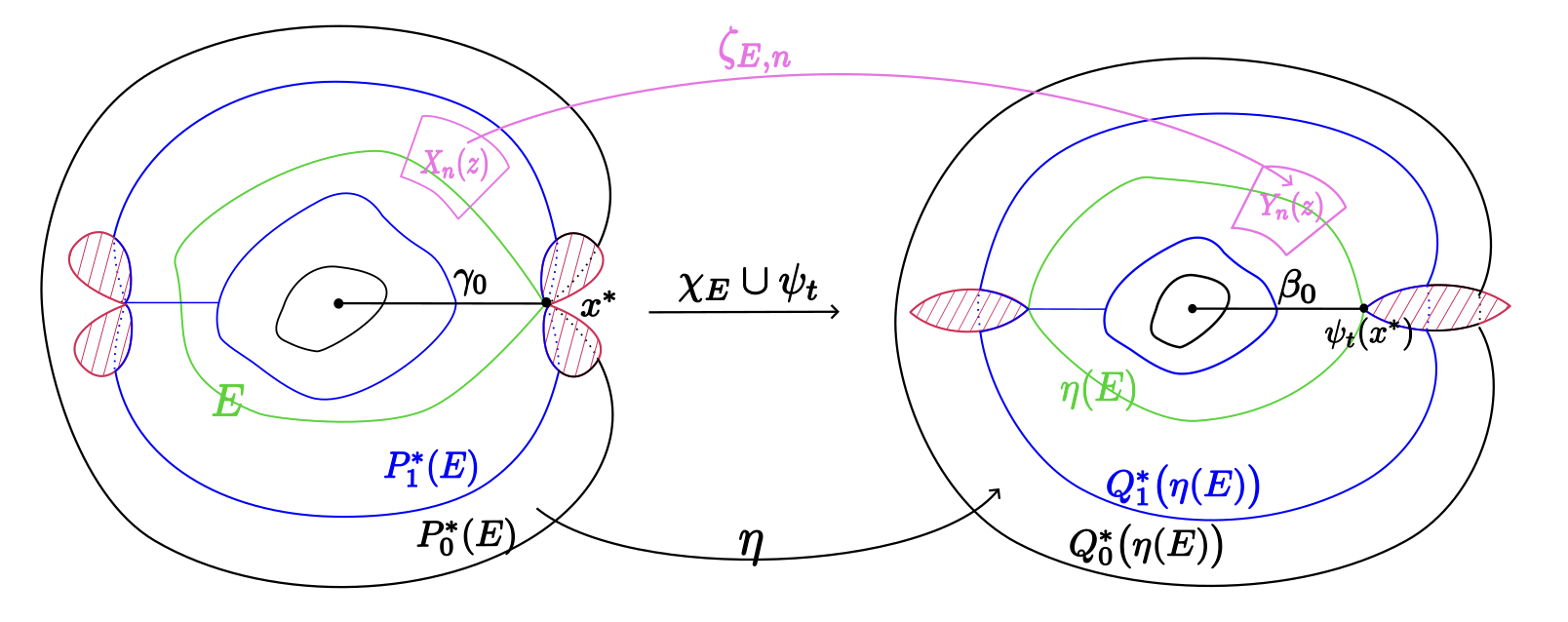}
\caption{}
\label{fig:7}
\end{figure}

If $E$ does not contain $x^*$, then by Claim \ref{claim:6}, there exists an open arc $\alpha_0 \subset U_f \setminus \bigcup_{k \geq 0} \overline{S_k}$ that lands on a fixed point $y = y(E) \in E$ and satisfies $\alpha_0 \subset f(\alpha_0)$. Recall that $\psi_t$ is univalent on $U_f \setminus \bigcup_{k \geq 0} \overline{S_k}$ and satisfies $\psi_t \circ f = f_t \circ \psi_t$. Define $\widetilde{\alpha}_0 = \psi_t(\alpha_0)$; this is an open arc in $U_{f_t} \cap Q_0^*(\eta(E))$ with $f_t(\widetilde{\alpha}_0) \supset \widetilde{\alpha}_0$. Using the same argument as in Claim \ref{claim:2}, we can show that $\widetilde{\alpha}_0$ lands at a fixed point $\widetilde{y} \in \eta(E)$ of $f_t$. In this case, we choose $\chi_E$ such that $\chi_E(y) = \widetilde{y}$ (see Figure \ref{fig:8}).

According to B\"{o}ttcher's Theorem \cite{M} to $\Omega_{E}$, there exists a conformal map $\varphi_1: \Omega_{E} \to \mathbb{D}$ such that $\varphi_1(a) =0$, and $\varphi_1\big(f(z)\big) = \big(\varphi_1(z)\big)^{d_0}$, where $a$ is the superattracting fixed point. Lifting the ray starting from $0$ and landing on $\partial\mathbb{D}$ via $\varphi_1$ yields a ray in $\Omega_{E}$, called an \textbf{internal ray}. By means of internal rays, we can construct partitions of puzzles $\mathcal{P}^*_n$ and $\mathcal{Q}^*_n$, denoted $\mathcal X_n(E)$ and $\mathcal Y_n(E)$, respectively. \vskip 0.1cm		

Define a ray $\g_0\subset P_0^*(E)$ as follows: if $x^*\in E$, define $\g_0$ to be the internal ray in $\Omega_E$ landing at $x^*$; otherwise, let $\g_0$ be the union of $\alpha_0$ and the closure of the internal ray in $\Omega_E$ landing at $y(E)$. Then $f(\g_0)\supset \g_0$ and $P_0^*(E)\setminus \g_0$ is a simply-connected domain containing no critical points and critical values of $f$.

For saving the notations, we define the map $\chi_E\cup \psi_t: \overline{\Omega_E}\bigcup (U_f\sm \bigcup_{k\geq0} \overline{S_k})\to\mathbb C$  as
		\[\big(\chi_E\cup \psi_t\big)(z)=\left\{
		\begin{array}{ll}
			\chi_E(z), & \hbox{if $z\in \overline{\Omega_E}$;} \\
			\psi_t(z), & \hbox{if $z\in U_f\sm \bigcup_{k\geq0} \overline{S_k}$.}
		\end{array}
		\right.
		\]
By the choice of $\chi_E$,  we obtain that $\beta_0:=\big(\chi_E\cup \psi_t\big)(\g_0)$ is a ray in $Q_0^*(\eta(E))$ joining the center of $\Omega_{\eta(E)}$ and a boundary point of $Q_0^*(\eta(E))$ such that $\beta_0\subset f_t(\beta_0)$.
\begin{figure}[htbp]
	\centering
	\includegraphics[width=16cm]{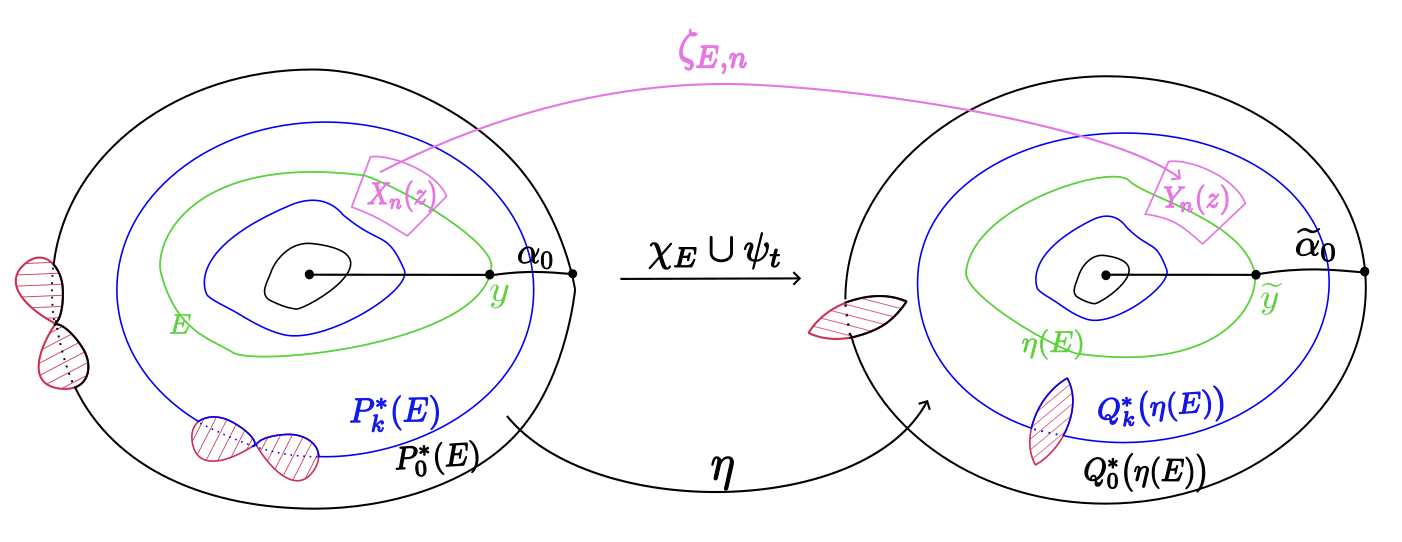}
	\caption{}
	\label{fig:8}
\end{figure}
		
Set $X_0:=P_0^*(E)\setminus(\overline{\Omega_E(1/2)}\cup \g_0)$.
For each $n \geq 1$, let $\mathcal X_n(E)$ denote the components of $f^{-n}(X_0)$ which intersect $E$, or equivalently, the components of $$P^*_{n}(E)\sm \big(\overline{\Omega_E(1/2d_0^n)}\cup f^{-n}(\g_0)\big),$$
where $d_0={\rm deg}(f|_E)$. Then each element of $\mathcal X_n(E)$  is a disk.
		
Let $z_*$ denote the unique intersection point of $\overline{\g_0}$ and $E$. If $z \in E \setminus \bigcup_{i \geq 0} f^{-i}(z_*)$, then it is contained in  a unique component of $\mathcal X_n(E)$,  denoted by $X_n(z)$, for each $n\geq1$. If $z \in E \cap f^{-i_0}(z_*)$ for some $i_0\geq0$, there are two adjacent components of $\mathcal X_n(E)$ such that $z$ belongs to their common boundary, for each $n > i_0$. In this case, denote $X_n(z)$ the union of the two components. It is worth noting that $\overline{X_n(z)}$ contains all points of the Julia set in a neighborhood of $z$.
		
Similarly, we can define puzzles $\mathcal Y_n(\eta(E))$ for $(f_t,\eta(E))$. Set $Y_0:=Q_0^*(\eta(E))\setminus(\overline{\Omega_{\eta(E)}(1/2)}\cup \beta_0)$.
For each $n \geq 1$, let $\mathcal Y_n(E)$ denote the components of $f_t^{-n}(Y_0)$ which intersect $\eta(E)$. For any $w\in \eta(E)$, the puzzle piece $Y_n(z)$ is  also  similarly defined as above. 
		
Since $\chi_E\cup\psi_t:\partial X_0\to\partial Y_0$ is a homeomorphism, and both $\chi_E$ and $\psi_t$ are conjugations between $f$ and $f_t$ on the corresponding definition domains, it follows that $$\chi_E\cup\psi_t:\bigcup_{X\in \mathcal X_n(E)} \partial{X}\longrightarrow \bigcup_{Y\in \mathcal Y_n(\eta(E))}\partial Y$$ is a homeomorphism for every $n\geq1$, and
\[(\chi_E\cup\psi_t)\circ f(w)=f_t\circ (\chi_E\cup\psi_t)(w),\ \, w\in \bigcup_{X\in \mathcal X_n(E)} \partial{X}.\]
Then the map $\chi_E\cup\psi_t$ induces a bijection $\zeta_{E,n}: \mathcal X_n(E)\to \mathcal Y_n(E)$ defined by
\[\zeta_{E,n}(X_n)=Y_n\text{ if }(\chi_E\cup\psi_t)(\partial X_n)=\partial Y_n\text{ (see Figures \ref{fig:7} and \ref{fig:8})},\]
and it holds that
	\begin{equation}\label{eq:00}
			\text{ $\zeta_{E,n-1}\circ f(X_n)=f_t\circ\zeta_{E,n}(X_n) \quad \text{for all }  X_n\in\mathcal X_n(E)$.}
		\end{equation}
		
Suppose now that $E'\in \mathcal E_f$ is preperiodic, i.e., $f^k(E')=E$ is periodic for some $k\geq1$. By lifting $\chi_E$ under $f^k$ and $f_t^k$, we obtain a conformal map $\chi_{E'}:\Omega_{E'}\to \Omega_{\eta(E')}$. Let $z(E')\in E'$ be a point such that $f^k(z(E))$ is the unique intersection of $\overline{\g_0}$ and $E$. Then there is a unique lift $\chi_{E'}$ of $\chi_E$ making $\chi_{E'}\cup \psi_t$ continuous at $z(E')$. This is the specific one for $E'$.
		
Then for every $n\geq0$, we define $\mathcal X_n(E')$ to be the collection of components of $f^{-k}(\mathcal X_n(E))$ that intersect $E'$, and  $\mathcal Y_n(\eta(E'))$ to be the collection of components of $f_t^{-n}(\mathcal Y_n(\eta(E)))$ that intersect $\eta(E')$. Moreover, the bijection $\zeta_{E,n}:\mathcal X_n(E)\to\mathcal Y_n(\eta(E))$ is lifted to a bijection $\zeta_{E',n}:\mathcal X_n(E')\to\mathcal Y_n(\eta(E'))$.
For any $z$ belongs to $ E'$ or $\eta(E')$, the set $X_n(z)$ or $Y_n(z)$ is defined similar to the periodic case.
		
\begin{claim}\label{claim:5}
For any $E \in \mathcal{E}_f$, $z \in E$, and $w \in \eta(E)$, we have:
\[\left\{
\begin{array}{llll}
	\bigcap_{n \geq 1} \overline{X_n(z)} = \{z\}, & \hbox{if $z \in E \setminus \bigcup_{i \geq 0} f^{-i}(z_*)$;} \\[3pt]
	\bigcap_{n > i_0} \overline{X_n(z)} = \{z\}, & \hbox{if $z \in E \cap f^{-i_0}(z_*)$ for some $i_0 \geq 0$;} \\[3pt]
	\bigcap_{n \geq 1} \overline{Y_n(w)} = \{w\}, & \hbox{if $w \in \eta(E) \setminus \bigcup_{i \geq 0} f_t^{-i}(\psi_t(z_*))$;} \\[3pt]
	\bigcap_{n > i_0} \overline{Y_n(w)} = \{w\}, & \hbox{if $w \in \eta(E) \cap f_t^{-i_0}(\psi_t(z_*))$ for some $i_0 \geq 0$.} 
\end{array}
\right. \]

\end{claim}
\begin{proof}
Suppose $z \in E \setminus \bigcup_{i \geq 0} f^{-i}(z_*)$. There exists an integer $k_0 \geq 1$ such that $X_1(z) \setminus \overline{X_{1+k_0}(z)}$ is an annulus. Choose integers $n_1 > 1 + k_0$ and $k_1 \geq 1$ for which $X_{n_1}(z) \setminus \overline{X_{n_1 + k_1}(z)}$ is also an annulus. 

If $k_1 = k_0$, the map
\[f^{n_1 - 1}: X_{n_1}(z) \setminus \overline{X_{n_1 + k_0}(z)} \to X_1(z) \setminus \overline{X_{1 + k_0}(z)}\]
is conformal, since $X_1(z)$ contains no critical points. These two annuli therefore have the same modulus. If $k_1 < k_0$, then $X_{n_1}(z) \setminus \overline{X_{n_1 + k_0}(z)}$ is also an annulus and has the same modulus as $X_1(z) \setminus \overline{X_{1 + k_0}(z)}$. If $k_1 > k_0$, then $X_{n_1}(z) \setminus \overline{X_{n_1 + k_1}(z)}$ has a larger modulus than $X_1(z) \setminus \overline{X_{1 + k_0}(z)}$.

Next, choose integers $n_2 > n_1 + \max\{k_0, k_1\}$ and $k_2 \geq 1$ such that $X_{n_2}(z) \setminus \overline{X_{n_2 + k_2}(z)}$ is also an annulus. Repeating this process, we obtain a nested sequence $X_1(z) \supset X_{n_1}(z) \supset \dots$, where each annulus $X_{n_i}(z) \setminus \overline{X_{n_i + \max\{k_{i-1}, k_i\}}(z)}$ has a modulus greater than or equal to that of $X_1(z) \setminus \overline{X_{1 + k_0}(z)}$.\vskip 0.1cm		

By Grötzsch's inequality,
\[\operatorname{mod}\Big( X_1(z) \setminus \bigcap_{n \geq 1} \overline{X_n(z)} \Big) \geq \sum_{i=1}^{\infty} \operatorname{mod}\left( X_{n_i}(z) \setminus \overline{X_{n_i + \max\{k_{i-1}, k_i\}}(z)} \right) = +\infty.\]
It follows that $\bigcap_{n \geq 1} \overline{X_n(z)} = \{z\}$.\vskip 0.1cm		
		
Now we assume $z \in E \cap f^{-i_0}(z_*)$ for some $i_0 \geq 0$. If $z \notin \bigcup_{\ell \geq 0} f^{-m}(x^*)$, then the proof that $\bigcap_{n \geq 1} \overline{X_n(z)} = \{z\}$ is identical to the case above. Suppose instead that $z \in f^{-m_0}(x^*)$ for some $m_0 \geq 0$. We first consider the subcase $z = x^*$. The following proof is based on \cite[Proposition 2]{Z3} by considering the local dynamics of the parabolic fixed point. 
	
Choose an integer $n_0 \geq 0$ such that $\partial X_n(x^*)$ contains only the fixed point $x^*$ for all $n \geq n_0$. Consider the holomorphic map
\[h := f^{-1}: X_{n_0}(x^*) \to X_{n_0+1}(x^*) \subset X_{n_0}(x^*),\]
which extends continuously to the boundary $\partial X_{n_0}(x^*)$. By \cite[Lemma 5.5]{M}, the iterates $h^n$ converge uniformly to the unique boundary fixed point $x^*$ on every compact subset of $X_{n_0}(x^*)$.

Select an integer $k_0 \geq n_0$ and a neighborhood $D$ of $x^*$ such that $\big( \overline{X_{k_0}(x^*)} \setminus \{x^*\} \big) \cap D$ is contained in the repelling petal at $x^*$. For any $y \in \overline{X_{k_0}(x^*)} \setminus \{x^*\}$, we may shrink $D$ (if necessary) to ensure $y \notin D$. Let 
\[B_1 := \overline{X_{k_0}(x^*)} \setminus D \quad \text{ and} \quad B_2 := \big( \overline{X_{k_0}(x^*)} \setminus \{x^*\} \big) \cap D.\]
 Since $B_1$ is a compact subset of $X_{n_0}(x^*)$, there exists a sufficiently large integer $M$ such that $h^n(B_1) \cap B_1 = \emptyset$ for all $n \geq M$. This implies $y \notin \bigcap_{n \geq M} h^n(B_1)$. Additionally, since $B_2$ lies in the repelling petal at $x^*$, we have $y \notin \bigcap_{n \geq M} h^n(B_2)$. Combining these results,
\[y \notin \Big( \bigcap_{n \geq M} h^n(B_1) \big) \cup \Big( \bigcap_{n \geq M} h^n(B_2) \big) = \bigcap_{n \geq M} h^n\Big( \overline{X_{k_0}(x^*)} \setminus \{x^*\} \Big) = \bigcap_{n \geq 1} \overline{X_n(x^*)} \setminus \{x^*\}.\]
It follows that $\bigcap_{n \geq 1} \overline{X_n(x^*)} = \{x^*\}$.

For any $z \in E \cap f^{-i_0}(x^*)$, we have
\[\bigcap_{n > i_0} \overline{X_n(z)} \subset \bigcap_{n > i_0} \overline{X_n\big(f^{-i_0}(x^*)\big)} = f^{-i_0}\Big( \bigcap_{n > i_0} \overline{X_n(x^*)} \Big) = f^{-i_0}(x^*).\]
Since $f^{-i_0}(x^*)$ is a set of discrete points, this implies $\bigcap_{n > i_0} \overline{X_n(z)} = \{z\}$.\vskip 0.1cm		

By the same reasoning, the conclusion stated in the claim holds for any $w \in \eta(E)$.
\end{proof}
		
Now, we can define a map $\phi:J(f)\to J(f_t)$  based on the puzzles constructed above.
		
For any $x\in J(f)$, it belongs to  an element $E$ of $\mathcal E_f$. By Lemma \ref{lem:component}, $E$ is a singleton if and only if $\eta(E)$ is a singleton. Combining  Claim \ref{claim:5},  we define $\phi:J(f)\to J(f_t)$ to be
		\[\phi(x)=\left\{
		\begin{array}{lll}
			\bigcap_{n\geq0}\overline{Q_n^*(\eta(E))}, & \hbox{if $E$ is a singleton;} \\[3pt]
			\bigcap_{n\geq1}\overline{\zeta_{E,n}(X_n(x))}, & \hbox{if $x\in E\sm\bigcup_{i \geq 0}f^{-i}(z_*)$;} \\[3pt]
			\bigcap_{n> i_0}\overline{\zeta_{E,n}(X_n(x))}, & \hbox{if $x\in E\cap f^{-i_0}(z_*)$ for some $i_0\geq 0$.}
		\end{array}
		\right.\]
		
Since $\eta:\mathcal E_f\to \mathcal E_{f_t}$ and $\zeta_{E,n}:\mathcal X_n(E)\to \mathcal Y_n(E),E\in\mathcal E_f,n\geq1$, are bijections, it is easy to verify that $\phi_t:J(f)\to J(f_t)$ is a bijection. According to formulas \eqref{eq:99} and \eqref{eq:00}, we also have that $\phi_t\circ f=f_t\circ\phi_t$  on $J(f)$. So it remains to prove the continuity of $\phi$.\vskip 0.1cm
		
Fix a point $x_{0} \in E$. If $E$ is a singleton, then $\text{diam}\big(\overline{Q^*_n(\eta(E))}\big)\to 0$. For any $x\in J(f)$ such that $x \to x_0$, we have $x \in \overline{P^*_M(E)}$ for some large enough integer $M$, hence $\phi(x) \in \overline{Q^*_M\big(\eta(E)\big)}$. Thus $\phi(x)\to\phi(x_0)$. If $E$ is not a singleton and fixed, then $\text{diam}\big(\overline{Y_n(\phi(x_0))}\big)\to 0$. For any $x\in J(f)$ such that $x\to x_0$, we have $x\in \overline{X_M(x_0)}$ for some large enough integer $M$, hence $\phi(x) \in \overline{Y_M\big(\phi(x_0)\big)}$. Thus $\phi(x)\to\phi(x_0)$. By the definition of $\phi$, it is easy to see that $\phi$ is continuous on $\sigma_{f}^{-n}(E)$ for any $n\geq 1$.
\end{proof}
	
The proof of Theorem \ref{thm:1.2}  requires two rigidity results about Cantor Julia sets \cite{Z1,Z2}.
	
\begin{theoremB}{\rm(\cite{Z1})}\label{Z1}
Let $f$ and $\widetilde{f}$ be two rational maps with  Cantor Julia sets. If they are topologically conjugate on $\omC$, then they are quasiconformally conjugate on $\omC$.
\end{theoremB}
	
\begin{theoremC}{\rm(\cite{Z2})}\label{Z2}
Let $f$ be a rational map with a Cantor Julia set. Then $f$ carries no invariant line fields on its Julia set.
\end{theoremC}

\begin{proof}[Proof of Theorem \ref{thm:1.2}]
		
Let $f$ be a rational map with a Cantor Julia set of degree $d$ and a parabolic fixed point. Then $\mathcal E_f$ contains no periodic critical elements.
		
By Propositions \ref{model} and \ref{pro:parabolic}, there exists a sequence of simple attracting maps $\{f_{t_j}\}_{j \geq 1}$ that converges uniformly on $\widehat{\mathbb{C}}$ to a simple parabolic map $g$. Moreover, $(f, U_f)$ and $(g, U_g)$ are conjugate via a conformal map $\psi: U_f \to U_g$ that fixes three points in $U_f$. For each $j \geq 1$, the map $\sigma_f: \mathcal{E}_f \to \mathcal{E}_f$ is conjugate both to $\sigma_{f_{t_j}}: \mathcal{E}_{f_{t_j}} \to \mathcal{E}_{f_{t_j}}$ and to $\sigma_g: \mathcal{E}_g \to \mathcal{E}_g$. Since $\mathcal{E}_f$ contains no periodic critical elements, neither do $\{\mathcal{E}_{f_{t_j}}\}_{j \geq 1}$ nor $\mathcal{E}_g$. We thus conclude from Lemma \ref{lem:component} that all maps in $\{f_{t_j}\}_{j \geq 1}$ and the map $g$ have Cantor Julia sets.
		
According to Theorem \ref{thm:middle}, it is enough to prove that $f=g$.
		
As in Section \ref{sec:2.1}, let $U_0(f) \subset U_f$ denote a regular parabolic petal of $f$, and let $U_n(f)$ denote the component of $f^{-n}(U_0(f))$ containing $U_0(f)$ for each $n \geq 0$. There exists $N > 0$ such that $\deg\big(f: U_N(f) \to U_{N-1}(f)\big)=d$. For each $n \geq 0$, the depth-$n$ puzzle $\mathcal{P}_n(f)$ for $f$ is defined as the collection of all components of $\mathbb{C} \setminus U_{N+n}(f)$.

By the conformal conjugation $\psi$, we define $U_n(g) = \psi(U_n(f))$ for $n \geq 0$ and $\mathcal{P}_n(g)$ as the collection of all components of $\mathbb{C} \setminus U_{N+n}(g)$. Note that $\psi: \overline{U_n(f)} \to \overline{U_n(g)}$ is a homeomorphism for each $n \geq 0$. For any $P_n \in \mathcal{P}_n(f)$, there exists a unique puzzle piece $Q_n \in \mathcal{P}_n(g)$ satisfying $\psi(\partial P_n) = \partial Q_n$, which may be written as $\widehat{\psi(\partial P_n)}$.

Using the correspondence between the puzzles $\mathcal{P}_n(f)$ and $\mathcal{P}_n(g)$, we extend $\psi$ to the Julia set $J(f)$ as follows. For any $x \in J(f)$, there exists a unique puzzle piece $P_n \in \mathcal{P}_n(f)$ for each $n$ such that $x \in P_n$. Then $\bigcap_{n \geq 0} P_n = \{x\}$. Define
\[\psi(x) = \bigcap_{n \geq 0} \widehat{\psi(\partial P_n)}.\]
This yields a global map $\psi: \widehat{\mathbb{C}} \to \widehat{\mathbb{C}}$.

By the shrinking property of the puzzle sequences $\{\mathcal{P}_n(f)\}_{n \geq 1}$ and $\{\mathcal{P}_n(g)\}_{n \geq 1}$, we can show that this extension $\psi$ is a homeomorphism on $\widehat{\mathbb{C}}$. The argument is similar to that in the proof of Theorem \ref{thm:middle}, so we omit the details. Moreover, the continuity of $\psi$ implies $\psi \circ f = g \circ \psi$ on $\widehat{\mathbb{C}}$, hence $\psi: \widehat{\mathbb{C}} \to \widehat{\mathbb{C}}$ is a topological conjugation between $f$ and $g$. We conclude from Theorem B that $\psi$ is a quasiconformal conjugation.
		
Since $\psi$  is conformal on $F(f)$ and has no invariant line field on $J(f)$ by Theorem C, it follows that the complex dilatation of $\psi$ is equal to $0$ almost everywhere in $\omC$. Consequently, $\psi$ is a conformal conjugation between $f$ and $g$. By the normalization property,  $\psi$ is the identity map, and therefore $f=g$.
\end{proof}
	
\vspace{3mm}

\noindent Ning Gao \\ Academy of Mathematics and Systems Science,\\ Chinese Academy of Sciences, Beijing, 100190, P. R. China.\\ gaoning@amss.ac.cn\vskip 0.24cm

\noindent Yan Gao\\School of Mathematical Sciences, \\Shenzhen University, Shenzhen, 518052, P. R. China.\\gyan@szu.edu.cn	\vskip 0.24cm

\noindent Wenjuan Peng \\State Key Laboratory of Mathematical Sciences, Academy of Mathematics and Systems Science,\\Chinese Academy of Sciences, Beijing, 100190, P. R. China.\\wenjpeng@amss.ac.cn

\begin{thebibliography}{WI}
		{\small
			\bibitem{Lectures} L. Ahlfors, \textit {Lectures on Quasiconformal Mappings, 2nd edn., volume 38 of University Lecture Series}, American Mathematical Society, Providence, RI (2006), with additional chapters by C. J. Earle, I. Kra, M. Shishikura and J. H. Hubbard.
			
			\bibitem{BH} B. Branner and J. H. Hubbard,  \textit{The iteration of cubic polynomials, \uppercase\expandafter{\romannumeral 2}.  Patterns and parapatterns}, Acta Math., 169 (1992), 229-325.
			
			\bibitem{CP} G. Cui and W. Peng, \textit{On the structure of Fatou domains}, Sci. China Ser. A-Math., 51(2008), No. 7, 1167-1186.
			
			\bibitem{CT} G. Cui and L. Tan, \textit{Hyperbolic-parabolic deformations of rational maps}, Sci. China: Math., 61 (2018), No. 12, 2157-2220.
			
			\bibitem{F} P. Fatou. \textit{Sur les équations fonctionnelles}, Bull. Soc. Math. France 48 (1920), 208–314. 
			
			\bibitem{GM} L. Goldberg and J. Milnor, \textit{Fixed points of polynomial maps. II. Fixed point portraits}, Ann. Sci. École Norm. Sup., 26 (1993), No. 1, 51–98.
			
			\bibitem{Hai} P. Haïssinsky, \textit{Déformation J-équivalente de polynômes géometriquement finis}, Fund. Math., 163 (2000), No. 2, 131-141.
			
			\bibitem{H} M. Herman, \textit{Sur la conjugaison différentiable des difféomorphismes du cercle à des rotations}, Publications Mathématiques de l'IHÉS, 49 (1979), 5–234.
			
			\bibitem{Ka1} T. Kawahira, \textit{Semiconjugacies between the Julia sets of geometrically finite rational maps}, Ergodic Theory Dynam. Systems, 23 (2003), No. 4, 1125–1152.
			
			\bibitem{Ka2} T. Kawahira, \textit{Semiconjugacies between the Julia sets of geometrically finite rational maps II}, European Mathematical Society (EMS), Zürich, 2006, 131–138.
			
			\bibitem{L} M. Lyubich and Y. Minsky, \textit{Laminations in holomorphic dynamics}, J. Differential Geom., 47 (1997), 17-94.
			
			\bibitem{M} J. Milnor, \textit{Dynamics in One Complex Variable}, 3$^{rd}$ edition, Princeton University Press, 2006.
			
			\bibitem{Mc} C. T. McMullen, \textit{Automorphisms of rational maps}, In: Drasin, Earle, Gehring, et al, eds, Holomorphic Functions and Moduli I, New York: Springer, 1998, 31-60.
			
			\bibitem{PYZ} W. Peng, Y. Yin and Y. Zhai, \textit{Density of hyperbolicity for rational maps with Cantor Julia sets}, Ergodic Theory  Dynam. Systems, 32 (2012), No. 5, 1711-1726.
			
			\bibitem{QY} W. Qiu and Y. Yin, \textit{Proof of the Branner-Hubbard conjecture on Cantor Julia sets}, Sci. China Ser. A., 52 (2009), No. 1, 45-65.
			
			\bibitem{S} C. L. Siegel, \textit{Iteration of analytic functions}, Ann.  Math., 43 (1942), 607–612. 
			
			\bibitem{Sul} D. P. Sullivan, \textit{Quasiconformal homeomorphisms and dynamics, \uppercase\expandafter{\romannumeral 1}. Solution of the Fatou-Julia problem on wandering domains}, Ann. Math., 122 (1985), No. 3, 401-418.
			
			\bibitem{TY} L. Tan and Y. Yin, \textit{Local connectivity of the Julia set for geometrically finite rational maps}, Sci. China Ser. A., 39 (1996), No. 1, 39-47.
			
			\bibitem{Z1} Y. Zhai, \textit{Rigidity for rational maps with Cantor Julia sets}, Sci. China Ser. A-Math., 51 (2008), No. 1, 79–92.
			
			\bibitem{Z2} Y. Yin and Y. Zhai, \textit{No invariant line fields on Cantor Julia sets}, Forum Math., 22 (2010), No. 1, 75–94.
			
			\bibitem{Z3} Y. Zhai, \textit{A generalized version of Branner-Hubbard conjecture for rational functions}, Acta Math. Sinica, 26 (2010), No. 11, 2199-2208.
		}
\end{thebibliography}
\end{document}